\documentclass[12pt,oneside,reqno,english]{amsart}
\usepackage[T1]{fontenc}
\usepackage[latin9]{inputenc}
\usepackage{babel}
\usepackage{array}
\usepackage{refstyle}
\usepackage{float}
\usepackage{units}
\usepackage{mathrsfs}
\usepackage{mathtools}
\usepackage{enumitem}
\usepackage{amstext}
\usepackage{amsthm}
\usepackage{amssymb}
\usepackage{stmaryrd}
\usepackage{graphicx}
\PassOptionsToPackage{normalem}{ulem}
\usepackage{ulem}
\usepackage[unicode=true,pdfusetitle,
 bookmarks=true,bookmarksnumbered=false,bookmarksopen=false,
 breaklinks=false,pdfborder={0 0 0},pdfborderstyle={},backref=false,colorlinks=false]
 {hyperref}
\hypersetup{
 colorlinks=true,citecolor=blue,linkcolor=blue,linktocpage=true,pdfstartview={XYZ null null 1.00},pdfpagemode=UseOutlines}

\makeatletter

\AtBeginDocument{\providecommand\lemref[1]{\ref{lem:#1}}}
\AtBeginDocument{\providecommand\thmref[1]{\ref{thm:#1}}}
\AtBeginDocument{\providecommand\figref[1]{\ref{fig:#1}}}
\AtBeginDocument{\providecommand\propref[1]{\ref{prop:#1}}}
\AtBeginDocument{\providecommand\corref[1]{\ref{cor:#1}}}
\AtBeginDocument{\providecommand\secref[1]{\ref{sec:#1}}}
\AtBeginDocument{\providecommand\remref[1]{\ref{rem:#1}}}
\AtBeginDocument{\providecommand\defref[1]{\ref{def:#1}}}
\providecommand{\tabularnewline}{\\}
\RS@ifundefined{subsecref}
  {\newref{subsec}{name = \RSsectxt}}
  {}
\RS@ifundefined{thmref}
  {\def\RSthmtxt{theorem~}\newref{thm}{name = \RSthmtxt}}
  {}
\RS@ifundefined{lemref}
  {\def\RSlemtxt{lemma~}\newref{lem}{name = \RSlemtxt}}
  {}

\numberwithin{equation}{section}
\numberwithin{figure}{section}
\numberwithin{table}{section}
\theoremstyle{plain}
\newtheorem{thm}{\protect\theoremname}[section]
\theoremstyle{plain}
\newtheorem{lem}[thm]{\protect\lemmaname}
\theoremstyle{definition}
\newtheorem{defn}[thm]{\protect\definitionname}
\theoremstyle{plain}
\newtheorem{prop}[thm]{\protect\propositionname}
\theoremstyle{remark}
\newtheorem{rem}[thm]{\protect\remarkname}
\theoremstyle{plain}
\newtheorem{cor}[thm]{\protect\corollaryname}
\theoremstyle{definition}
\newtheorem{example}[thm]{\protect\examplename}
\theoremstyle{remark}
\newtheorem*{acknowledgement*}{\protect\acknowledgementname}

\allowdisplaybreaks
\usepackage{needspace}

\usepackage{refstyle}
\usepackage{etoolbox}

\newref{lem}{refcmd={Lemma \ref{#1}}}
\newref{thm}{refcmd={Theorem \ref{#1}}}
\newref{cor}{refcmd={Corollary \ref{#1}}}
\newref{sec}{refcmd={Section \ref{#1}}}
\newref{sub}{refcmd={Section \ref{#1}}}
\newref{subsec}{refcmd={Section \ref{#1}}}
\newref{chap}{refcmd={Chapter \ref{#1}}}
\newref{prop}{refcmd={Proposition \ref{#1}}}
\newref{exa}{refcmd={Example \ref{#1}}}
\newref{tab}{refcmd={Table \ref{#1}}}
\newref{rem}{refcmd={Remark \ref{#1}}}
\newref{def}{refcmd={Definition \ref{#1}}}
\newref{fig}{refcmd={Figure \ref{#1}}}

\setlist[enumerate,1]{label=(\roman*),ref=\roman*}
\setlist[enumerate,2]{label=(\alph*),ref=\theenumi \alph*}

\@ifundefined{showcaptionsetup}{}{%
 \PassOptionsToPackage{caption=false}{subfig}}
\usepackage{subfig}
\AtBeginDocument{
  
}

\makeatother

\providecommand{\acknowledgementname}{Acknowledgement}
\providecommand{\corollaryname}{Corollary}
\providecommand{\definitionname}{Definition}
\providecommand{\examplename}{Example}
\providecommand{\lemmaname}{Lemma}
\providecommand{\propositionname}{Proposition}
\providecommand{\remarkname}{Remark}
\providecommand{\theoremname}{Theorem}

\begin{document}
\title[]{On reproducing kernels, and analysis of measures}
\dedicatory{To the memory of J\o rgen Hoffmann-J\o rgensen\textsuperscript{$1$}.}
\begin{abstract}
Starting with the correspondence between positive definite kernels
on the one hand and reproducing kernel Hilbert spaces (RKHSs) on the
other, we turn to a detailed analysis of associated measures and Gaussian
processes. Point of departure: Every positive definite kernel is also
the covariance kernel of a Gaussian process.

Given a fixed sigma-finite measure $\mu$, we consider positive definite
kernels defined on the subset of the sigma algebra having finite $\mu$
measure. We show that then the corresponding Hilbert factorizations
consist of signed measures, finitely additive, but not automatically
sigma-additive. We give a necessary and sufficient condition for when
the measures in the RKHS, and the Hilbert factorizations, are sigma-additive.
Our emphasis is the case when $\mu$ is assumed non-atomic. By contrast,
when $\mu$ is known to be atomic, our setting is shown to generalize
that of Shannon-interpolation. Our RKHS-approach further leads to
new insight into the associated Gaussian processes, their It\^{o}
calculus and diffusion. Examples include fractional Brownian motion,
and time-change processes.
\end{abstract}

\author{Palle Jorgensen and Feng Tian}
\address{(Palle E.T. Jorgensen) Department of Mathematics, The University of
Iowa, Iowa City, IA 52242-1419, U.S.A. }
\email{palle-jorgensen@uiowa.edu}
\urladdr{http://www.math.uiowa.edu/\textasciitilde jorgen/}
\address{(Feng Tian) Department of Mathematics, Hampton University, Hampton,
VA 23668, U.S.A.}
\email{feng.tian@hamptonu.edu}
\keywords{Hilbert space, reproducing kernel Hilbert space, harmonic analysis,
Gaussian free fields, transforms, covariance.\\
\textsuperscript{1}J\o rgen Hoffmann-J\o rgensen (1942-2017) was
a pioneer in probability theory; adding there to the subject both
profound insight and originality. He made several fundamental contributions,
including a result, now known as the Hoffmann-J\o rgensen's inequality.
He was a visionary pioneer in high-dimensional probability theory;
and he was a former colleague, in Aarhus, of the first named author.}

\maketitle
\tableofcontents{}

\section{Introduction}

A reproducing kernel Hilbert space (RKHS) is a Hilbert space $\mathscr{H}$
of functions (defined on a prescribed set) in which point-evaluation
is a continuous linear functional; so continuity is required to hold
with respect to the norm in $\mathscr{H}$. These Hilbert spaces (RKHS)
have a host of applications, including to complex analysis, to harmonic
analysis, and to quantum mechanics.

A fundamental theorem of Aronszajn yields an explicit correspondence
between positive definite kernels on the one hand and RKHSs on the
other. Now every positive definite kernel is also the covariance kernel
of a Gaussian process; a fact which is a point of departure in our
present analysis. Given a positive definite kernel, we shall explore
its use in the analysis of the associated Gaussian process; and vice
versa.

This point of view is especially fruitful when one is dealing with
problems from stochastic analysis. Even restricting to stochastic
analysis, we have the exciting area of applications to statistical
learning theory \cite{MR2327597,MR3236858}. The RKHSs are useful
in statistical learning theory on account of a powerful representer
theorem: It states that every function in an RKHS that minimizes an
associated empirical risk-function can be written as a generalized
linear combination of samplings of the kernel function; i.e., samples
evaluated at prescribed training points. Hence, it is a popular tool
for empirical risk minimization problems, as it adapts perfectly to
a host of infinite dimensional optimization problems.

Analysis with the use of reproducing kernel Hilbert space (RKHS) has
found diverse applications in many areas. However, presently we shall
focus on applications to probability theory; applications to such
important and related topics as metric entropy computations, to small
deviation problems for Gaussian processes, and to i.i.d. series representations
for general classes of Gaussian processes. We refer to a detailed
discussion of these items below, with citations.

Recall that a \emph{reproducing kernel Hilbert space} (RKHS) is a
Hilbert space $\mathscr{H}$ of functions, say $f$, on a fixed set
$X$ such that every linear functional (induced by $x\in X$),  
\begin{equation}
E_{x}\left(f\right):=f\left(x\right),\quad f\in\mathscr{H}.\label{eq:A1}
\end{equation}
is continuous in the norm of $\mathscr{H}$. 

Hence, by Riesz' representation theorem, there is a corresponding
$h_{x}\in\mathscr{H}$ such that 
\begin{equation}
E_{x}f=\left\langle f,h_{x}\right\rangle _{\mathscr{H}}\label{eq:I2}
\end{equation}
where $\left\langle \cdot,\cdot\right\rangle _{\mathscr{H}}$ denotes
the inner product in $\mathscr{H}$. Setting 
\[
K\left(x,y\right)=\left\langle h_{y},h_{x}\right\rangle _{\mathscr{H}},\quad\left(x,y\right)\in X\times X
\]
we get a \emph{positive definite} (p.d.) kernel, i.e., $\forall n\in\mathbb{N}$,
$\forall\left\{ \alpha_{i}\right\} _{1}^{n}$, $\forall\left\{ x_{i}\right\} _{1}^{n}$,
$\alpha_{i}\in\mathbb{C}$, $x_{i}\in X$, we have 
\begin{equation}
\sum_{i}\sum_{j}\alpha_{i}\overline{\alpha}_{j}K\left(x_{i},x_{j}\right)\geq0.\label{eq:I3}
\end{equation}

Conversely, if $K$ is given p.d., i.e., satisfying (\ref{eq:I3}),
then by \cite{MR0051437}, there is a RKHS such that (\ref{eq:I2})
holds.

Given $K$ p.d., we may take $\mathscr{H}\left(K\right)$ to be the
completion of 
\begin{equation}
\psi=\sum_{i}\alpha_{i}K\left(\cdot,x_{i}\right)\label{eq:I4}
\end{equation}
in the norm 
\begin{equation}
\left\Vert \psi\right\Vert _{\mathscr{H}\left(K\right)}^{2}=\sum_{i}\sum_{j}\alpha_{i}\overline{\alpha}_{j}K\left(x_{i},x_{j}\right),
\end{equation}
but quotiented out by those functions $\psi$ in (\ref{eq:I4}) with
$\left\Vert \psi\right\Vert _{\mathscr{H}\left(K\right)}^{2}=0$.
(In fact, by \lemref{I1} below, $\left\Vert \psi\right\Vert _{\mathscr{H}\left(K\right)}=0$
implies that $\psi\left(x\right)=0$, for all $x\in X$.)

A key fact which we shall be using throughout the paper is the following: 
\begin{lem}
\label{lem:I1}Let $K$ be a positive definite kernel on $X\times X$,
and let $\mathscr{H}\left(K\right)$ be the corresponding RKHS. 

Then a function $f$ on $X$ is in $\mathscr{H}\left(K\right)$ \uline{iff}
there is a finite constant $C=C_{f}$, depending on $f$, such that
$\forall n\in\mathbb{N}$, $\forall\left\{ x_{i}\right\} _{1}^{n}$,
$\left\{ \alpha_{i}\right\} _{1}^{n}$, $x_{i}\in X$, $\alpha_{i}\in\mathbb{C}$,
we have: 
\begin{equation}
\left|\sum\nolimits _{i=1}^{n}\alpha_{i}f\left(x_{i}\right)\right|^{2}\leq C\sum\nolimits _{i}\sum\nolimits _{j}\alpha_{i}\overline{\alpha}_{j}K\left(x_{i},x_{j}\right).\label{eq:I6}
\end{equation}
\end{lem}

\begin{proof}[Proof sketch (for the benefit of the readers)]
One direction in the proof is immediate from the following observation
regarding the norm $\left\Vert \cdot\right\Vert _{\mathscr{H}\left(K\right)}$
in the RKHS $\mathscr{H}\left(K\right)$. Here a reproducing kernel
$K$ is fixed: For all finite sums, $\alpha_{i}\in\mathbb{C}$, $x_{i}\in X$,
$1\leq i\leq N$, we then have:
\[
\left\Vert \sum\nolimits _{i=1}^{N}\alpha_{i}K\left(x_{i},\cdot\right)\right\Vert _{\mathscr{H}\left(X\right)}^{2}=\sum\nolimits _{i=1}^{N}\sum\nolimits _{j=1}^{N}\alpha_{i}\overline{\alpha_{j}}K\left(x_{i},x_{j}\right),
\]
i.e., the RKHS in (\ref{eq:I6}).

Now assume a function $f$ on $X$ is given to satisfy (\ref{eq:I6}).
Then define a linear functional $T_{f}$ on $\mathscr{H}\left(K\right)$,
as follows: First define it on the above finite linear combinations
(recall dense in $\mathscr{H}\left(K\right)$): 
\[
T_{f}\left(\sum\nolimits _{i=1}^{N}\alpha_{i}K\left(x_{i},\cdot\right)\right)=\sum\nolimits _{i=1}^{N}\alpha_{i}f\left(x_{i}\right).
\]
The assumption (\ref{eq:I6}) simply amounts to the following \emph{a
priori} estimate:
\begin{equation}
\left|T_{f}\left(\psi\right)\right|^{2}\leq C\left\Vert \psi\right\Vert _{\mathscr{H}\left(K\right)}^{2}\label{eq:A7}
\end{equation}
where $\psi$ has the form of (\ref{eq:I4}). Since, by (\ref{eq:A7}),
$T_{f}$ defines a bounded linear functional on a dense subspace in
$\mathscr{H}\left(K\right)$, it extends by limits (in the $\mathscr{H}\left(K\right)$-norm)
to $\mathscr{H}\left(K\right)$. So by Riesz' lemma (for Hilbert spaces)
applied to $\mathscr{H}\left(K\right)$, we get the stated inner-product
representation
\[
T_{f}\left(\psi\right)=\left\langle F,\psi\right\rangle _{\mathscr{H}\left(K\right)}
\]
for a unique $F\in\mathscr{H}\left(K\right)$. Using again the reproducing
property (\ref{eq:A1})-(\ref{eq:I2}) for $\left\langle \cdot,\cdot\right\rangle _{\mathscr{H}\left(K\right)}$
(inner product), we conclude that $F=f$ holds (pointwise identity)
for the two functions; hence $f\in\mathscr{H}\left(K\right)$. 
\end{proof}
Our present core theme, is motivated by, and makes direct connections
to, a number of areas in probability theory. For the benefit of readers,
we add below some hints to a number of such important and related
topics, metric entropy, small deviation problems for Gaussian processes,
and series representations of Gaussian processes. Of special note
are the following three:

\textbf{1.} Metric entropy of the unit ball of the RKHS of Gaussian
measures/processes. We refer to the fundamental papers by Dudley and
Sudakov \cite{MR0220340,MR0247034}. For a reformulation of their
results in functional-analytic terms see \cite{MR695735}.

\textbf{2.} Small ball problems for Gaussian measures on Banach spaces/small
deviation problems for Gaussian processes. Kuelbs and Li achieved
a breakthrough in this area \cite{MR1237989}. Further relevant contributions
(including also fractional Brownian motion) can be found e.g. in \cite{MR1733160}
and \cite{MR1724026}.

\textbf{3.} Series representations of Gaussian processes (similar
to Karhunen-Lo\'eve expansions). See e.g., \cite{MR1935652}.

\section{\label{sec:SigA}Sigma-algebras and RKHSs of signed measures}

Now our present focus will be a class of p.d. kernels, defined on
subsets of a fixed $\sigma$-algebra. Specifically, if $\left(M,\mathscr{B},\mu\right)$
is a $\sigma$-finite measure space, we set $X=\mathscr{B}_{fin}$;
see (\ref{eq:S1}) below.
\begin{defn}
Consider a measure space $\left(M,\mathscr{B},\mu\right)$ where $\mathscr{B}$
is a sigma-algebra of subsets in $M$, and $\mu$ is a $\sigma$-finite
measure on $\mathscr{B}$. Set 
\begin{equation}
\mathscr{B}_{fin}=\left\{ A\in\mathscr{B}\mid\mu\left(A\right)<\infty\right\} .\label{eq:S1}
\end{equation}
Let $\mathscr{H}$ be a Hilbert space having the following property:
\begin{equation}
\left\{ \chi_{A}\mid A\in\mathscr{B}_{fin}\right\} \subset\mathscr{H},\label{eq:S2}
\end{equation}
where $\chi_{A}$ denotes the indicator function for the set $A$.
\end{defn}

We shall restrict the discussion to real valued functions, and the
extension to the complex case is straightforward. The latter can be
found in a number of treatments, for example Peres et al. \cite{MR2552864}.
\begin{thm}
\label{thm:S2}Let $\beta$ be a function, $\mathscr{B}_{fin}\times\mathscr{B}_{fin}\longrightarrow\mathbb{R}$.
Then TFAE:
\begin{enumerate}
\item \label{enu:S1}$\beta$ is positive definite, i.e., $\forall n\in\mathbb{N}$,
$\forall\left\{ \alpha_{i}\right\} _{1}^{n}$, $\forall\left\{ A_{i}\right\} _{1}^{n}$,
$\alpha_{i}\in\mathbb{R}$, $A_{i}\in\mathscr{B}_{fin}$, we have
\begin{equation}
\sum_{1}^{n}\sum_{1}^{n}\alpha_{i}\alpha_{j}\beta\left(A_{i},A_{j}\right)\geq0.
\end{equation}
\item \label{enu:S2}There is a Hilbert space $\mathscr{H}$ which satisfies
(\ref{eq:S2}); and also 
\begin{equation}
\beta\left(A,B\right)=\left\langle \chi_{A},\chi_{B}\right\rangle _{\mathscr{H}},\quad\forall\left(A,B\right)\in\mathscr{B}_{fin}\times\mathscr{B}_{fin}.\label{eq:S4}
\end{equation}
\item \label{enu:S3}There is a Hilbert space $\mathscr{H}$ which satisfies
(\ref{eq:S2}); and also a linear mapping: 
\begin{equation}
\mathscr{H}\ni f\longmapsto\mu_{f}\in\left(\begin{matrix}\text{signed finitely additive}\\
\text{measures on \ensuremath{\left(M,\mathscr{B}\right)}}
\end{matrix}\right)\label{eq:S5}
\end{equation}
with 
\begin{equation}
\mu_{f}\left(A\right)=\left\langle \chi_{A},f\right\rangle _{\mathscr{H}},\quad\forall A\in\mathscr{B}_{fin}.\label{eq:S6}
\end{equation}
\end{enumerate}
\end{thm}

\begin{proof}
We shall divide up the reasoning in the implications: (\ref{enu:S1})
$\Rightarrow$ (\ref{enu:S2}) $\Rightarrow$ (\ref{enu:S3}) $\Rightarrow$
(\ref{enu:S1}). The characterization in (\ref{eq:S5}) of the elements
in the RKHS $\mathscr{H}$ is based on an application of \lemref{I1},
combined with the detailed reasoning below.

\textbf{Case (\ref{enu:S1}) $\Rightarrow$ (\ref{enu:S2}).} Given
a function $\beta$ as in (\ref{enu:S1}), we know that, by \cite{MR0051437},
there is an associated reproducing kernel Hilbert space $\mathscr{H}\left(\beta\right)$.
The vectors in $\mathscr{H}\left(\beta\right)$ are obtained by the
quotient and completion procedures applied to the functions 
\begin{equation}
\mathscr{B}_{fin}\ni B\longmapsto\beta\left(A,B\right)\in\mathbb{R}\label{eq:S7}
\end{equation}
defined for every $A\in\mathscr{B}_{fin}$. Moreover, the inner product
in $\mathscr{H}\left(\beta\right)$, satisfies 
\begin{equation}
\left\langle \beta\left(A_{1},\cdot\right),\beta\left(A_{2},\cdot\right)\right\rangle _{\mathscr{H}\left(\beta\right)}=\beta\left(A_{1},A_{2}\right).
\end{equation}

Now let 
\begin{equation}
\mathscr{H}=\big(span\left\{ \chi_{A}\mid A\in\mathscr{B}_{fin}\right\} \big)^{\sim}
\end{equation}
with $\left(\cdots\right)^{\sim}$ denoting the Hilbert completion:
\begin{eqnarray*}
\left\Vert \sum\nolimits _{i}\alpha_{i}\chi_{A_{i}}\right\Vert _{\mathscr{H}}^{2} & = & \sum\nolimits _{i}\sum\nolimits _{j}\alpha_{i}\alpha_{j}\beta\left(A_{i},A_{j}\right)\\
 & \underset{\text{see \ensuremath{\left(\ref{eq:S7}\right)}}}{=} & \left\Vert \sum\nolimits _{i=1}^{n}\alpha_{i}\beta\left(A_{i},\cdot\right)\right\Vert _{\mathscr{H}\left(\beta\right)}^{2}.
\end{eqnarray*}
It is then immediate from this that the Hilbert space $\mathscr{H}$
satisfies the conditions stated in (\ref{enu:S2}) of the theorem. 

\textbf{Case (\ref{enu:S2}) $\Rightarrow$ (\ref{enu:S3}).} Let
$\mathscr{H}$ satisfy the conditions in (\ref{enu:S2}); and for
$f\in\mathscr{H}$, let $\mu_{f}$ be as in (\ref{eq:S5}). We must
show that if $n\in\mathbb{N}$, $\left\{ A_{i}\right\} _{1}^{n}$,
$A_{i}\in\mathscr{B}_{fin}$, satisfy $A_{i}\cap A_{j}=\emptyset$,
$i\neq j$, then 
\begin{equation}
\mu_{f}\left(\cup_{1}^{n}A_{i}\right)=\sum\nolimits _{1}^{n}\mu_{f}\left(A_{i}\right).\label{eq:S10}
\end{equation}
But 
\begin{eqnarray*}
\text{LHS}_{\left(\ref{eq:S10}\right)} & \underset{\text{by \ensuremath{\left(\ref{eq:S6}\right)}}}{=} & \left\langle \chi_{\cup_{i=1}^{n}A_{i}},f\right\rangle _{\mathscr{H}}\\
 & = & \sum_{i}\left\langle \chi_{A_{i}},f\right\rangle _{\mathscr{H}}=\sum_{i}\mu_{f}\left(A_{i}\right)=\text{RHS}_{\left(\ref{eq:S10}\right)}.
\end{eqnarray*}
The remaining assertions in (\ref{enu:S3}) are clear. 

\textbf{Case (\ref{enu:S3}) $\Rightarrow$ (\ref{enu:S1}).} This
step is immediate from (\ref{eq:S4}).
\end{proof}
\begin{prop}
Let $\left(M,\mathscr{B},\mu\right)$ be a $\sigma$-finite measure
space. As in \thmref{S2}, we specify a pair $\left(\beta,\mathscr{H}\right)$
where $\beta$ is defined on $\mathscr{B}_{fin}\times B_{fin}$, and
$\mathscr{H}$ is a Hilbert space subject to condition (\ref{eq:S2}).
For $f\in\mathscr{H}$, set 
\begin{equation}
\mu_{f}\left(A\right)=\left\langle \chi_{A},f\right\rangle _{\mathscr{H}},\quad A\in\mathscr{B}_{fin}.
\end{equation}
Then $\mu_{f}\in\mathscr{H}\left(\beta\right)\left(=\text{the RKHS of \ensuremath{\beta}.}\right)$
Moreover, 
\begin{equation}
\left\Vert \mu_{f}\right\Vert _{\mathscr{H}\left(\beta\right)}\leq\left\Vert f\right\Vert _{\mathscr{H}}.\label{eq:I12}
\end{equation}
\end{prop}

\begin{proof}
This will be a direct application of \lemref{I1}, but now applied
to $X=\mathscr{B}_{fin}$. Hence we must show that, $\forall n\in\mathbb{N}$,
$\left\{ A_{i}\right\} _{1}^{n}$, $\left\{ \alpha_{i}\right\} _{1}^{n}$,
$A_{i}\in\mathscr{B}_{fin}$, $\alpha_{i}\in\mathbb{R}$, the estimate
(\ref{eq:I6}) holds, and with a finite constant $C_{f}$.

In fact, we may take $C_{f}=\left\Vert f\right\Vert _{\mathscr{H}}^{2}$,
so $\left\Vert \mu_{f}\right\Vert _{\mathscr{H}\left(\beta\right)}\leq\left\Vert f\right\Vert _{\mathscr{H}}$
as claimed. Specifically, 
\begin{eqnarray*}
\left|\sum\nolimits _{i=1}^{n}\alpha_{i}\mu_{f}\left(A_{i}\right)\right|^{2} & = & \left|\sum\nolimits _{i=1}^{n}\alpha_{i}\left\langle \chi_{A_{i}},f\right\rangle _{\mathscr{H}}\right|^{2}\\
 & = & \left|\left\langle \sum\nolimits _{i=1}^{n}\alpha_{i}\chi_{A_{i}},f\right\rangle _{\mathscr{H}}\right|^{2}\\
 & \underset{{\scriptscriptstyle \text{by Schwarz}}}{\leq} & \left\Vert \sum\nolimits _{i=1}^{n}\alpha_{i}\chi_{A_{i}}\right\Vert _{\mathscr{H}}^{2}\left\Vert f\right\Vert _{\mathscr{H}}^{2}\\
 & \underset{{\scriptscriptstyle \text{by \ensuremath{\left(\ref{eq:S4}\right)}}}}{=} & \left\Vert f\right\Vert _{\mathscr{H}}^{2}\sum\nolimits _{i}\sum\nolimits _{j}\alpha_{i}\alpha_{j}\beta\left(A_{i},A_{j}\right),
\end{eqnarray*}
which is the desired conclusion. 
\end{proof}
\begin{rem}
Our present focus is on the case when the prescribed $\sigma$-finite
measure $\mu$ is non-atomic. But the atomic case is also important,
for example in interpolation theory in the form of Shannon, see e.g.,
\cite{MR0442564}. 

Consider, for example, the case $X=\mathbb{R}$, and 
\begin{equation}
K\left(x,y\right)=\frac{\sin\pi\left(x-y\right)}{\pi\left(x-y\right)},\label{eq:I7}
\end{equation}
defined for $\left(x,y\right)\in\mathbb{R}\times\mathbb{R}$. In this
case, the RKHS $\mathscr{H}\left(K\right)$ is familiar: It may be
realized as functions $f$ on $\mathbb{R}$, such that the Fourier
transform 
\begin{equation}
\hat{f}\left(\xi\right)=\int_{\mathbb{R}}e^{-i2\pi x\xi}f\left(x\right)dx\label{eq:I8}
\end{equation}
is well defined, \emph{and} supported in the compact interval $\left[-\frac{1}{2},\frac{1}{2}\right]$,
frequency band, with $\left\Vert f\right\Vert _{\mathscr{H}\left(K\right)}^{2}=\int_{-\frac{1}{2}}^{\frac{1}{2}}|\hat{f}\left(\xi\right)|^{2}d\xi.$ 

Set $\mu=\sum_{n\in\mathbb{Z}}\delta_{n}$ (the Dirac-comb). Then
Shannon's theorem states that 
\begin{equation}
l^{2}\left(\mathbb{Z}\right)\ni\left(\alpha_{n}\right)_{n\in\mathbb{Z}}\xrightarrow{\quad T\quad}\mathscr{H}\left(K\right),
\end{equation}
given by 
\begin{equation}
\big(T\left(\left(\alpha_{n}\right)\big)\right)\left(x\right)=\sum_{n\in\mathbb{Z}}\alpha_{n}\,\frac{\sin\pi\left(x-n\right)}{\pi\left(x-n\right)}
\end{equation}
is isometric, mapping $l^{2}$ onto $\mathscr{H}\left(K\right)$.
Its adjoint operator
\[
T^{*}:\mathscr{H}\left(K\right)\longrightarrow l^{2}\left(\mathbb{Z}\right)
\]
is 
\begin{equation}
\left(T^{*}f\right)_{n}=f\left(n\right),\quad n\in\mathbb{Z}.\label{eq:I11}
\end{equation}
Compare (\ref{eq:I11}) with (\ref{eq:T6}) below in a much wider
context. 

The RKHS for the kernel (\ref{eq:I7}) $\mathscr{H}\left(K\right)$
is called the Paley-Wiener space. Functions in $\mathscr{H}\left(K\right)$
also go by the name, band-limited signals. We refer to (\ref{eq:I11})
as (Shannon) sampling. It states that functions (continuous time-signals)
$f$ from $\mathscr{H}\left(K\right)$ may be reconstructed \textquotedblleft perfectly\textquotedblright{}
from their discrete $\mathbb{Z}$ samples. 
\end{rem}

\section{\label{sec:SA}The sigma-additive property}

The sigma-additive property alluded to here is not a minor technical
point. Indeed, one of the basic problems related to the propositional
calculus and the foundations of quantum mechanics is the description
of probability measures (called states in quantum physical terminology)
on the set of experimentally verifiable propositions. In the quantum
setting, the set of propositions is then realized as an orthomodular
partially ordered set, where the order is induced by a relation of
implication, called a quantum logic. Now quantum-observables are generally
non-commuting, and the precise question is in fact formulated for
states (measures) on $C^{*}$-algebras; i.e., normalized positive
linear functionals (see e.g., \cite{MR3642406}).

The classical Gleason theorem (see \cite{MR0096113}) is the assertion
that a state on the $C^{*}$-algebra $\mathscr{B}\left(\mathscr{H}\right)$
of all bounded operators on a Hilbert space is uniquely described
by the values it takes on orthogonal projections, assuming the dimension
of the Hilbert space $\mathscr{H}$ is not 2. The precise result entails
extension of finitely additive measures to sigma-additive counterparts,
i.e., when we have additivity on countable unions of disjoint sets
from the underlying sigma-algebra.

We now turn to the question of when the finitely additive measures
$\mu_{f}$ are in fact $\sigma$-additive. (See \thmref{S2}, part
(\ref{enu:S3}).)

Given $\left(M,\mathscr{B},\mu\right)$ as above, we shall set 
\begin{equation}
\mathscr{D}_{fin}\left(\mu\right)=span\left\{ \chi_{A}\mid A\in\mathscr{B}_{fin}\right\} .
\end{equation}
Recall that $\mathscr{D}_{fin}\left(\mu\right)$ is automatically
a dense subspace in $L^{2}\left(\mu\right)$. 
\begin{thm}
\label{thm:T1}Let $\mathscr{B}_{fin}$ be as specified in (\ref{eq:S1})
with a fixed $\sigma$-finite measure space $\left(M,\mathscr{B},\mu\right)$.
Let $\beta$ be given, assumed positive definite on $\mathscr{B}_{fin}\times\mathscr{B}_{fin}$,
and let $\mathscr{H}$ be a Hilbert space which satisfies conditions
(\ref{eq:S2}) and (\ref{eq:S4}). 

Then there is a dense subspace $\mathscr{H}_{\mu}\subset\mathscr{H}$
such that the signed measures
\begin{equation}
\left\{ \mu_{f}\mid f\in\mathscr{H}_{\mu}\right\} 
\end{equation}
are $\sigma$-additive if and only if the following implication holds:
\begin{alignat*}{1}
\left.\begin{matrix}\left(\alpha\right) &  & \left\{ \varphi_{n}\right\} _{n\in\mathbb{N}},\:\varphi_{n}\in\mathscr{D}_{fin}\left(\mu\right),\:\left\Vert \varphi_{n}\right\Vert _{L^{2}\left(\mu\right)}\xrightarrow[\;n\rightarrow\infty\;]{}0\\
\left(\beta\right) &  & f\in\mathscr{H},\;\left\Vert \varphi_{n}-f\right\Vert _{\mathscr{H}}\xrightarrow[\;n\rightarrow\infty\;]{}0
\end{matrix}\right\} \Longrightarrow & f=0,
\end{alignat*}
i.e., if a vector $f\in\mathscr{H}$ satisfies $\left(\alpha\right)$
and $\left(\beta\right)$, it must be the null vector in $\mathscr{H}$. 
\end{thm}

\begin{proof}
Note that, because of assumptions (\ref{eq:S2}) and (\ref{eq:S4}),
we get a natural inclusion mapping, denoted $T$, 
\begin{equation}
L^{2}\left(\mu\right)\xrightarrow{\quad T\quad}\mathscr{H}\label{eq:T3}
\end{equation}
with dense domain $\mathscr{D}_{fin}\left(\mu\right)$ in $L^{2}\left(\mu\right)$.
Recall, if $A\in\mathscr{B}_{fin}$, then the indicator function $\chi_{A}$
is assumed to be in $\mathscr{H}$.

With these assumptions, we see that the implication in the statement
of the theorem simply states that $T$ is closable when viewed as
a densely defined operator as in (\ref{eq:T3}).

By a general theorem (see e.g., \cite{MR3642406}), $T$ is closable
if and only if the domain $dom\left(T^{*}\right)$ of its adjoint
$T^{*}$ is dense in $\mathscr{H}$. 

We have that a vector $f$ in $\mathscr{H}$ is in $dom\left(T^{*}\right)$
if and only if  $\exists C_{f}<\infty$ such that 
\begin{equation}
\left|\left\langle T\varphi,f\right\rangle _{\mathscr{H}}\right|\leq C_{f}\left\Vert \varphi\right\Vert _{L^{2}\left(\mu\right)}
\end{equation}
holds for all $\varphi\in\mathscr{D}_{fin}\left(\mu\right)$. Also
note that, if $\varphi=\chi_{A}$, $A\in\mathscr{B}_{fin}$, then
\begin{equation}
\left\langle T\varphi,f\right\rangle _{\mathscr{H}}=\mu_{f}\left(A\right);
\end{equation}
and so if $f\in dom\left(T^{*}\right)$, then 
\begin{align}
\mu_{f}\left(A\right) & =\big\langle\chi_{A},\underset{{\scriptscriptstyle \in L^{2}\left(\mu\right)}}{\underbrace{T^{*}f}}\big\rangle_{L^{2}\left(\mu\right)}=\int_{A}\left(T^{*}f\right)d\mu,\quad\forall A\in\mathscr{B}_{fin}.\label{eq:T6}
\end{align}
Note, by definition, $T^{*}f\in L^{2}\left(\mu\right)$. Indeed, the
converse holds as well. Since the right-hand side in (\ref{eq:T6})
is clearly $\sigma$-additive, one implication holds. Moreover, the
other implication follows from general facts about $L^{2}\left(M,\mathscr{B},\mu\right)$
valid for any $\sigma$-finite measure $\mu$ on $\left(M,\mathscr{B}\right)$. 
\end{proof}

\begin{cor}
\label{cor:T2}Let $\left(\beta,\mathscr{H}\right)$ be as in the
statement of \thmref{T1}, and let $T$ be the closable inclusion
$L^{2}\left(\mu\right)\xrightarrow{\;T\;}\mathscr{H}$. Then for $f\in dom\left(T^{*}\right)$,
dense in $\mathscr{H}$, the corresponding signed measure $\mu_{f}$
is absolutely continuous w.r.t. $\mu$ with Radon-Nikodym derivative
\begin{equation}
\frac{d\mu_{f}}{d\mu}=T^{*}f.\label{eq:T7}
\end{equation}
\end{cor}

\begin{example}
Let $\left(M,\mathscr{B},\mu\right)$ be a $\sigma$-finite measure
space, and on $\mathscr{B}_{fin}\times\mathscr{B}_{fin}$ define 
\begin{equation}
\beta_{\mu}\left(A,B\right):=\mu\left(A\cap B\right),\quad\forall A,B\in\mathscr{B}_{fin}.\label{eq:T8}
\end{equation}
Let $\mathscr{H}\left(\beta_{\mu}\right)=\text{RKHS}(\beta_{\mu})$,
i.e., the reproducing kernel Hilbert space associated with the p.d.
function $\beta_{\mu}$. Then $\mathscr{H}\left(\beta_{\mu}\right)$
consists of all signed measures $m$ of the form 
\begin{equation}
m\left(A\right)=\int_{A}\varphi\,d\mu,\quad\varphi\in L^{2}\left(\mu\right);\label{eq:T9}
\end{equation}
and when (\ref{eq:T9}) holds, 
\begin{equation}
\left\Vert m\right\Vert _{\mathscr{H}\left(\beta_{\mu}\right)}^{2}=\int_{M}\left|\varphi\right|^{2}d\mu.
\end{equation}
\end{example}

\begin{proof}
When $\beta_{\mu}$ is specified as in (\ref{eq:T8}), then one checks
immediately that the inclusion operator $T:L^{2}\left(\mu\right)\longrightarrow\mathscr{H}\left(\beta_{\mu}\right)$
is isometric, and maps onto $\mathscr{H}\left(\beta_{\mu}\right)$.
Indeed, for finite linear combinations $\sum_{i=1}^{n}\alpha_{i}\chi_{A_{i}}$
as above, we have 
\begin{align*}
\left\Vert \sum\nolimits _{i}\alpha_{i}\chi_{A_{i}}\right\Vert _{L^{2}\left(\mu\right)}^{2} & =\sum\nolimits _{i}\sum\nolimits _{j}\alpha_{i}\alpha_{j}\mu\left(A_{i}\cap A_{j}\right)\\
 & =\left\Vert \sum\nolimits _{i}\alpha_{i}\beta_{\mu}\left(A_{i},\cdot\right)\right\Vert _{\mathscr{H}\left(\beta_{\mu}\right)}^{2},
\end{align*}
so $T$ is isometric and onto. 
\end{proof}

\section{\label{sec:GF}Gaussian Fields}

Let $\left(M,\mathscr{B},\mu\right)$ be a $\sigma$-finite measure
space. By a \emph{Gaussian field} based on $\left(M,\mathscr{B},\mu\right)$,
we mean a probability space $\left(\Omega,\mathscr{C},\mathbb{P}^{\left(\mu\right)}\right)$,
depending on $\mu$, such that $\mathscr{C}$ is a $\sigma$-algebra
of subsets of $\Omega$, and $\mathbb{P}^{\left(\mu\right)}$ is a
probability measure on $\left(\Omega,\mathscr{C}\right)$. For every
$A\in\mathscr{B}_{fin}$, it is assumed that $X_{A}^{\left(\mu\right)}$
is in $L^{2}\left(\Omega,\mathscr{C},\mathbb{P}^{\left(\mu\right)}\right)$;
and in addition, 
\begin{equation}
X_{A}^{\left(\mu\right)}\sim N\left(0,\mu\left(A\right)\right),
\end{equation}
i.e., the distribution of $X_{A}^{\left(\mu\right)}$, computed for
$\mathbb{P}^{\left(\mu\right)}$ is the standard Gaussian with variance
$\mu\left(A\right)$. 

Finally, set $\mathbb{E}_{\mu}\left(\cdot\right)=\int_{\Omega}\left(\cdot\right)d\mathbb{P}^{\left(\mu\right)}$;
then it is required that 
\begin{equation}
\mathbb{E}_{\mu}\left(X_{A}^{\left(\mu\right)}X_{B}^{\left(\mu\right)}\right)=\mu\left(A\cap B\right),\quad\forall A,B\in\mathscr{B}_{fin}.\label{eq:G2}
\end{equation}
For a background reference on probability spaces, see e.g., \cite{MR1278486}.
\begin{prop}
\label{prop:G1}Given $\left(M,\mathscr{B},\mu\right)$, $\sigma$-finite,
then there is an associated Gaussian field $\{X_{A}^{\left(\mu\right)}\}_{A\in\mathscr{B}_{fin}}$
satisfying 
\begin{equation}
\mathbb{E}\left(X_{A}^{\left(\mu\right)}X_{B}^{\left(\mu\right)}\right)=\mu\left(A\cap B\right),
\end{equation}
for all $A,B\in\mathscr{B}_{fin}$.
\end{prop}

\begin{proof}
For all $n\in\mathbb{N}$, $\left\{ A_{i}\right\} _{1}^{n}$, $A_{i}\in\mathscr{B}_{fin}$,
let $g^{\left(A_{i}\right)}$ be the Gaussian distribution on $\mathbb{R}^{n}$,
with mean zero, and covariance matrix 
\begin{equation}
\big[\mu\left(A_{i}\cap A_{j}\right)\big]_{i,j=1}^{n}.\label{eq:G3}
\end{equation}
By Kolmogorov's theorem \cite{MR0032961,MR0150810,MR0279844,MR562914,MR3272038,MR3642406},
there is a unique probability measure $\mathbb{P}^{\left(\mu\right)}$
on the infinite Cartesian product
\begin{equation}
\Omega=\dot{\mathbb{R}}^{\mathscr{B}_{fin}}\label{eq:G4}
\end{equation}
such that 
\begin{equation}
\mathbb{E}_{\mu}\left(\cdot\cdot\mid\left\{ A_{1},\cdots,A_{n}\right\} \right)=g^{\left(A_{i}\right)}.
\end{equation}
For $\omega\in\Omega=\dot{\mathbb{R}}^{\mathscr{B}_{fin}}$, set 
\begin{equation}
X_{A}^{\left(\mu\right)}\left(\omega\right)=\omega\left(A\right),\quad A\in\mathscr{B}_{fin}.
\end{equation}
For the $\sigma$-algebra $\mathscr{C}$ of subsets in $\Omega$,
we take the cylinder $\sigma$-algebra, which is generated by 
\begin{equation}
\left\{ \omega\in\Omega\mid a_{i}<\omega\left(A_{i}\right)<b_{i}\right\} ,
\end{equation}
with $\left\{ A_{i}\right\} _{1}^{n}\subset\mathscr{B}_{fin}$, and
open intervals $\left(a_{i},b_{i}\right)$; see \figref{cyl}.
\end{proof}
\begin{figure}
\includegraphics[width=0.6\textwidth]{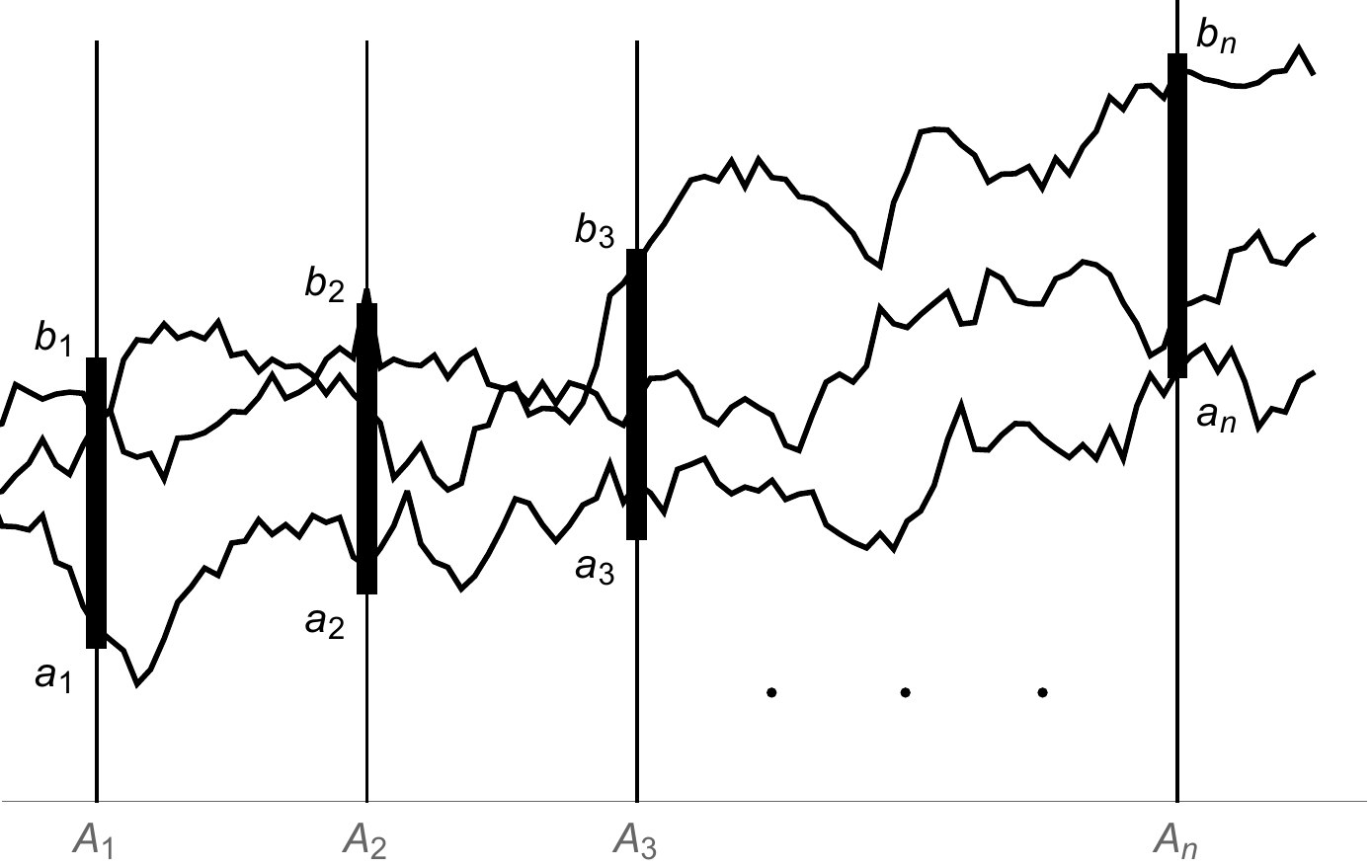}

\caption{\label{fig:cyl}A cylinder set in $\Omega$.}

\end{figure}

\begin{cor}
\label{cor:G2}Let $\left(M,\mathscr{B},\mu\right)$ be given, $\sigma$-finite,
and let $X^{\left(\mu\right)}$ be an associated Gaussian field; see
\propref{G1}, and (\ref{eq:G2}). 

Let $\mathscr{D}_{fin}\left(\mu\right)=span\left\{ \chi_{A}\mid A\in\mathscr{B}_{fin}\right\} $;
then 
\begin{equation}
\mathscr{D}_{fin}\left(\mu\right)\ni\sum_{i}\alpha_{i}\chi_{A_{i}}\longmapsto\sum_{i}\alpha_{i}X_{A_{i}}^{\left(\mu\right)}
\end{equation}
extends by closure to an isometry of $L^{2}\left(\mu\right)$ into
$L^{2}\left(\Omega,\mathbb{P}^{\left(\mu\right)}\right)$, called
the generalized It\^{o}-Wiener integral. 
\end{cor}

\begin{proof}
We have for all linear combinations as above, 
\begin{eqnarray*}
\left\Vert \sum\nolimits _{i}\alpha_{i}X_{A_{i}}^{\left(\mu\right)}\right\Vert _{L^{2}(\Omega,\mathbb{P}^{\left(\mu\right)})}^{2} & = & \sum\nolimits _{i}\sum\nolimits _{j}\alpha_{i}\alpha_{j}\mathbb{E}_{\mu}\left(X_{A_{i}}^{\left(\mu\right)}X_{A_{j}}^{\left(\mu\right)}\right)\\
 & \underset{\text{by \ensuremath{\left(\ref{eq:G2}\right)}}}{=} & \sum\nolimits _{i}\sum\nolimits _{j}\alpha_{i}\alpha_{j}\mu\left(A_{i}\cap A_{j}\right)\\
 & = & \left\Vert \sum\nolimits _{i}\alpha_{i}\chi_{A_{i}}\right\Vert _{L^{2}\left(\mu\right)}^{2}
\end{eqnarray*}
which is the desired isometry. Hence 
\begin{equation}
T_{\mu}:\underset{\varphi}{\underbrace{\sum\nolimits _{i}\alpha_{i}\chi_{A_{i}}}}\longrightarrow\sum\nolimits _{i}\alpha_{i}X_{A_{i}}^{\left(\mu\right)}
\end{equation}
extends by closure to an isometry 
\begin{equation}
T_{\mu}\left(\varphi\right):=X_{\varphi}^{\left(\mu\right)},\label{eq:G10}
\end{equation}
i.e., 
\[
\mathbb{E}_{\mu}\left(\left|X_{\varphi}^{\left(\mu\right)}\right|^{2}\right)=\int_{M}\left|\varphi\right|^{2}d\mu,\quad\text{and}\quad\mathbb{E}_{\mu}\left(X_{\varphi_{1}}^{\left(\mu\right)}X_{\varphi_{2}}^{\left(\mu\right)}\right)=\int_{M}\varphi_{1}\varphi_{2}\,d\mu
\]
hold for all $\varphi_{1},\varphi_{2}\in L^{2}\left(\mu\right)$.
Moreover, $X_{\varphi}^{\left(\mu\right)}\sim N\big(0,\left\Vert \varphi\right\Vert _{L^{2}\left(\mu\right)}^{2}\big)$
as stated. 
\end{proof}
\begin{cor}
\label{cor:G3}Let $\left(M,\mathscr{B},\mu\right)$ be as above,
i.e., $\mu$ is assumed $\sigma$-finite. Suppose, in addition, that
$\mu$ is non-atomic; then the quadratic variation of the Gaussian
process $X^{\left(\mu\right)}$ coincides with the measure $\mu$
itself. 
\end{cor}

\begin{proof}
Consider $B\in\mathscr{B}_{fin}$, and consider all partitions $PAR\left(B\right)$
of the set $B$, i.e., 
\begin{equation}
\pi=\left\{ \left(A_{i}\right)\right\} 
\end{equation}
specified as follows: $A_{i}\in\mathscr{B}_{fin}$, $A_{i}\cap A_{j}=\emptyset$
if $i\neq j$, and $\cup_{i}A_{i}=B$. 

We consider the limit over the net of such partitions.We show that
\begin{equation}
\mathbb{E}_{\mu}\left(\left|\mu\left(B\right)-\sum\nolimits _{i}(X_{A_{i}}^{\left(\mu\right)})^{2}\right|^{2}\right)\longrightarrow0\label{eq:G12}
\end{equation}
as $\pi\rightarrow0$, i.e., $\max_{i}\mu\left(A_{i}\right)\rightarrow0$,
for $\pi=\left(A_{i}\right)\in PAR\left(B\right)$. 

Since, for $\pi=\left(A_{i}\right)\in PAR\left(B\right)$, we have
$\sum_{i}\mu\left(A_{i}\right)=\mu\left(B\right)$, to prove (\ref{eq:G12}),
we need only consider the individual terms; $i$ fixed: 
\begin{eqnarray*}
 &  & \mathbb{E}_{\mu}\left(\left|\mu\left(A_{i}\right)-(X_{A_{i}}^{\left(\mu\right)})^{2}\right|^{2}\right)\\
 & = & \mu\left(A_{i}\right)^{2}-2\mu\left(A_{i}\right)\mathbb{E}_{\mu}\left((X_{A_{i}}^{\left(\mu\right)})^{2}\right)+\mathbb{E}_{\mu}\left((X_{A_{i}}^{\left(\mu\right)})^{4}\right).
\end{eqnarray*}
But 
\[
\mathbb{E}_{\mu}\left((X_{A_{i}}^{\left(\mu\right)})^{2}\right)=\mu\left(A_{i}\right),\quad\text{and}\quad\mathbb{E}_{\mu}\left((X_{A_{i}}^{\left(\mu\right)})^{4}\right)=3\mu\left(A_{i}\right)^{2};
\]
and so 
\[
\mathbb{E}_{\mu}\left(\left|\mu\left(A_{i}\right)-(X_{A_{i}}^{\left(\mu\right)})^{2}\right|^{2}\right)=2\mu\left(A_{i}\right)^{2}.
\]
Now, for $\pi=\left(A_{i}\right)\in PAR\left(B\right)$, we have:
\[
\sum_{i}\mu\left(A_{i}\right)^{2}\leq\underset{\rightarrow0}{\underbrace{\left(\max_{i}\mu\left(A_{i}\right)\right)}}\mu\left(B\right)\quad\text{as \ensuremath{\pi\rightarrow0;}}
\]
and the desired conclusion (\ref{eq:G12}) follows.

By general theory, fixing a non-atomic measure space $\left(\mathscr{B},\mu\right)$,
then the set $\pi$ of all $\left(\mathscr{B},\mu\right)$-partitions
(see above) can be given an obvious structure of refinement. This
in turn yields a corresponding net, and net-convergence refers limit
over this net, as the refinement mesh tends to zero. Specifically,
as $\max_{i}\mu\left(A_{i}\right)\rightarrow0$.
\end{proof}
\begin{cor}
Let $\mu$ and $\nu$ be two positive $\sigma$-finite measures on
a fixed measure space $\left(M,\mathscr{B}\right)$; see \corref{G3}
for the detailed setting. Let $X^{\left(\mu\right)}$ and $X^{\left(\nu\right)}$
be the corresponding Gaussian fields. Consider nets of partitions
$\pi=\left\{ \left(A_{i}\right)\right\} $ from $\left(M,\mathscr{B}\right)$. 
\begin{enumerate}
\item If $B\in\mathscr{B}$, then the limit
\begin{equation}
\lim_{\stackrel{\pi\rightarrow0}{{\scriptscriptstyle \pi\in PAR\left(B\right)}}}\sum_{i}X_{A_{i}}^{\left(\mu\right)}X_{A_{i}}^{\left(\nu\right)}
\end{equation}
exists; and it defines a signed measure, denoted $\langle X^{\left(\mu\right)},X^{\left(\nu\right)}\rangle$,
satisfying 
\begin{equation}
\langle X^{\left(\mu\right)},X^{\left(\nu\right)}\rangle=\frac{1}{2}\left(\langle X^{\left(\mu\right)}\rangle+\langle X^{\left(\nu\right)}\rangle-\langle X^{\left(\mu\right)}-X^{\left(\nu\right)}\rangle\right).
\end{equation}
\item If $\lambda$ is a positive measure on $\left(M,\mathscr{B}\right)$
satisfying $\mu\ll\lambda$, and $\nu\ll\lambda$, with respective
Radon-Nikodym derivatives $d\mu/d\lambda$ and $d\nu/d\lambda$, then
\begin{equation}
\langle X^{\left(\mu\right)},X^{\left(\nu\right)}\rangle=\sqrt{\frac{d\mu}{d\lambda}\frac{d\nu}{d\lambda}}\,d\lambda,\label{eq:G15}
\end{equation}
where the representation in (\ref{eq:G15}) is \uline{independent}
of the choice of measures $\lambda$ subject to: $\mu\ll\lambda$,
$\nu\ll\lambda$. 
\end{enumerate}
\end{cor}

\begin{proof}
The details follow those in the proof of \corref{G3} above; and we
also make use of the theory of sigma-Hilbert spaces (universal Hilbert
spaces); see e.g., \cite{MR0282379,zbMATH06897817,2018arXiv180506063J}.
\end{proof}
\begin{cor}
Let $\left(M,\mathscr{B},\mu\right)$, $X^{\left(\mu\right)}$, and
$T_{\mu}:L^{2}\left(\mu\right)\longrightarrow L^{2}(\mathbb{P}^{\left(\mu\right)})$
be as in \propref{G1}, then the adjoint 
\[
T_{\mu}^{*}:L^{2}(\Omega,\mathbb{P}^{\left(\mu\right)})\longrightarrow L^{2}\left(M,\mu\right)
\]
is specified as follows: 

Let $n\in\mathbb{N}$, and let $p\left(x_{1},x_{2},\cdots,x_{n}\right)$
be a polynomial on $\mathbb{R}^{n}$. For 
\begin{equation}
F:=p\left(X_{\varphi_{1}}^{\left(\mu\right)},\cdots,X_{\varphi_{n}}^{\left(\mu\right)}\right),\quad\left\{ \varphi_{i}\right\} _{1}^{n},\:\varphi_{i}\in L^{2}\left(\mu\right);\label{eq:G16}
\end{equation}
set 
\begin{equation}
D\left(F\right):=\sum_{i=1}^{n}\frac{\partial p}{\partial x_{i}}\left(X_{\varphi_{1}}^{\left(\mu\right)},\cdots,X_{\varphi_{n}}^{\left(\mu\right)}\right)\otimes\varphi_{i}.\label{eq:G17}
\end{equation}
Then we get the adjoint $T_{\mu}^{*}$ of the isometry $T_{\mu}$
expressed as: 
\begin{equation}
T_{\mu}^{*}\left(F\right)=\sum_{i=1}^{n}\mathbb{E}_{\mu}\left(\frac{\partial p}{\partial x_{i}}\left(X_{\varphi_{1}}^{\left(\mu\right)},\cdots,X_{\varphi_{n}}^{\left(\mu\right)}\right)\right)\varphi_{i}.\label{eq:G18}
\end{equation}
(Note that the right-hand side in (\ref{eq:G18}) is in $L^{2}\left(\mu\right)$.)
\end{cor}

\begin{proof}[Proof sketch]
Recall that 
\[
T_{\mu}\psi:=X_{\psi}^{\left(\mu\right)}:L^{2}\left(\mu\right)\longrightarrow L^{2}(\Omega,\mathbb{P}^{\left(\mu\right)})
\]
as in (\ref{eq:G10}), and 
\begin{equation}
X_{\psi}^{\left(\mu\right)}=\int_{M}\psi\,dX^{\left(\mu\right)}\label{eq:G19}
\end{equation}
is the stochastic integral, where $dX^{\left(\mu\right)}$ denotes
the It\^{o}-Wiener integral. 

The arguments combine the results in the present section, and standard
facts regarding the Malliavin derivative. (See, e.g., \cite{MR3630401,MR1952822,MR2382071,MR3501849}.)
Recall that the operator
\begin{equation}
D:L^{2}(\Omega,\mathbb{P}^{\left(\mu\right)})\longrightarrow L^{2}(\Omega,\mathbb{P}^{\left(\mu\right)})\otimes L^{2}\left(\mu\right)
\end{equation}
from (\ref{eq:G17}) is the Malliavin derivative corresponding to
the Gaussian field (\ref{eq:G19}); see also \corref{G2}.

In the arguments below, we restrict consideration to the case of real
valued functions. We shall also make use of the known fact that the
space of functions $F$ in (\ref{eq:G16}) is dense in $L^{2}(\Omega,\mathbb{P}^{\left(\mu\right)})$
as $n\in\mathbb{N}$, polynomials $p\left(x_{1},\cdots,x_{n}\right)$,
and $\left\{ \varphi_{i}\right\} _{1}^{n}$ vary, $\varphi_{i}\in L^{2}\left(\mu\right)$. 

The key step in the verification of the formula (\ref{eq:G18}) for
$T^{*}$, form $L^{2}(\Omega,\mathbb{P}^{\left(\mu\right)})$ onto
$L^{2}\left(\mu\right)$, is the following assertion: Let $F$ and
$X_{\psi}^{\left(\mu\right)}$, $\psi\in L^{2}\left(\mu\right)$,
be as stated; then 
\begin{align}
\left\langle F,X_{\psi}^{\left(\mu\right)}\right\rangle _{L^{2}(\Omega,\mathbb{P}^{\left(\mu\right)})} & =\mathbb{E}_{\mu}\left(FX_{\psi}^{\left(\mu\right)}\right)\nonumber \\
 & =\sum_{i=1}^{n}\mathbb{E}_{\mu}\left(\frac{\partial p}{\partial x_{i}}\left(X_{\varphi_{1}}^{\left(\mu\right)},\cdots,X_{\varphi_{n}}^{\left(\mu\right)}\right)\right)\left\langle \varphi_{i},\psi\right\rangle _{L^{2}\left(\mu\right)}.\label{eq:G421}
\end{align}
But (\ref{eq:G421}) in turn follows from the basic formula for the
finite-dimensional Gaussian distributions $g^{\left(n\right)}\left(x\right)$
in \propref{G1} above. We have: 
\begin{eqnarray*}
 &  & \int_{\mathbb{R}^{n}}\frac{\partial p}{\partial x_{i}}\left(x_{1},\cdots,x_{n}\right)g^{\left(n\right)}\left(x_{1},\cdots,x_{n}\right)d^{\left(n\right)}x\\
 & = & \int_{\mathbb{R}^{n}}x_{i}p\left(x_{1},\cdots,x_{n}\right)g^{\left(n\right)}\left(x_{1},\cdots,x_{n}\right)d^{\left(n\right)}x
\end{eqnarray*}
where $d^{\left(n\right)}x=dx_{1}dx_{2}\cdots dx_{n}$ is the standard
Lebesgue measure on $\mathbb{R}^{n}$. 

The general case is as follows: Set $C=\left[\mu\left(A_{i}\cap A_{j}\right)\right]_{i,j}$,
the covariance matrix from (\ref{eq:G3}), and 
\[
g\left(x\right):=g^{\left(A_{i}\right)}\left(x\right)=\left(\det C\right)^{-n/2}e^{-\frac{1}{2}\left\langle x,C^{-1}x\right\rangle _{\mathbb{R}^{n}}};
\]
then 
\begin{eqnarray*}
 &  & \mathbb{E}_{\mu}\left(\sum\nolimits _{i}\frac{\partial p}{\partial x_{i}}\left\langle \varphi_{i},\psi\right\rangle _{L^{2}\left(\mu\right)}\right)\\
 & = & \int_{\mathbb{R}^{n}}\sum\nolimits _{i}\frac{\partial p}{\partial x_{i}}\left(x\right)g\left(x\right)\left\langle \varphi_{i},\psi\right\rangle _{L^{2}\left(\mu\right)}dx^{\left(n\right)}\\
 & = & \int_{\mathbb{R}^{n}}p\left(x\right)\left(\sum\nolimits _{i,j}C_{ij}^{-1}x_{j}\right)g\left(x\right)\left\langle \varphi_{i},\psi\right\rangle _{L^{2}\left(\mu\right)}dx^{\left(n\right)}\\
 & = & \mathbb{E}_{\mu}\left(p\,T_{\mu}\left(\sum\nolimits _{i,j}C_{ij}^{-1}\varphi_{j}\left\langle \varphi_{i},\psi\right\rangle _{L^{2}\left(\mu\right)}\right)\right),
\end{eqnarray*}
where 
\[
\psi\longmapsto\sum_{i,j}C_{ij}^{-1}\varphi_{j}\left\langle \varphi_{i},\psi\right\rangle _{L^{2}\left(\mu\right)}
\]
is the projection from $\psi$ onto $span\left\{ \varphi_{i}\right\} $. 

Recall the correspondence $\left(p,\varphi_{1},\cdots,\varphi_{n}\right)\longleftrightarrow F$
in (\ref{eq:G16}), where $p=p\left(x_{1},\cdots,x_{n}\right)$, $x=\left(x_{1},\cdots,x_{n}\right)\in\mathbb{R}^{n}$.
The random variable $F$ has the Wiener-chaos representation in (\ref{eq:G16}). 
\end{proof}

\begin{cor}
Let $\left(M,\mathscr{B},\mu\right)$ be a $\sigma$-finite measure,
and let $\{X_{\varphi}^{\left(\mu\right)}\mid\varphi\in L^{2}\left(\mu\right)\}$
be the corresponding Gaussian field. We then have the following covariance
relations for $(X_{\varphi}^{\left(\mu\right)})^{m}$ corresponding
to the even and odd values of $m\in\mathbb{N}$: 
\begin{align*}
\mathbb{E}_{\mu}\left(\left(X_{\varphi}^{\left(\mu\right)}\right)^{2n}X_{\psi}^{\left(\mu\right)}\right) & =0,\quad\forall\varphi,\psi\in L^{2}\left(\mu\right);\;\text{and}\\
\mathbb{E}_{\mu}\left(\left(X_{\varphi}^{\left(\mu\right)}\right)^{2n+1}X_{\psi}^{\left(\mu\right)}\right) & =\left\langle \varphi,\psi\right\rangle _{L^{2}\left(\mu\right)}\left\Vert \varphi\right\Vert _{L^{2}\left(\mu\right)}^{2n}\left(2n+1\right)!!
\end{align*}
where
\[
\left(2n+1\right)!!=\left(2n+1\right)\left(2n-1\right)\cdots5\cdot3=\frac{\left(2\left(n+1\right)\right)!}{2^{n+1}\left(n+1\right)!}.
\]
\end{cor}

\begin{proof}
This is immediate from (\ref{eq:G421}), and an induction argument.
Take $n=1$, and $p\left(x\right)=x^{m}$; starting with 
\[
\mathbb{E}_{\mu}\left(\left(X_{\varphi}^{\left(\mu\right)}\right)^{2}X_{\psi}^{\left(\mu\right)}\right)=2\mathbb{E}_{\mu}\left(X_{\varphi}^{\left(\mu\right)}\right)\left\langle \varphi,\psi\right\rangle _{L^{2}\left(\mu\right)}=0
\]
and 
\[
\mathbb{E}_{\mu}\left(\left(X_{\varphi}^{\left(\mu\right)}\right)^{3}X_{\psi}^{\left(\mu\right)}\right)=3\underset{{\scriptscriptstyle \left\Vert \varphi\right\Vert _{L^{2}\left(\mu\right)}^{2}}}{\underbrace{\mathbb{E}_{\mu}\left(\left(X_{\varphi}^{\left(\mu\right)}\right)^{2}\right)}}\left\langle \varphi,\psi\right\rangle _{L^{2}\left(\mu\right)}.
\]
\end{proof}

\subsection{It\^{o} calculus}

In this section we discuss properties of the Gaussian process corresponding
to the Hilbert space factorizations from the setting in \thmref{T1}.

The initial setting is a fixed $\sigma$-finite measure space $\left(M,\mathscr{B},\mu\right)$
with corresponding 
\begin{equation}
\mathscr{B}_{fin}=\left\{ A\in\mathscr{B}\mid\mu\left(A\right)<\infty\right\} .\label{eq:GI1}
\end{equation}
As in \secref{SA}, we shall study positive definite (p.d.) functions
$\beta$
\begin{equation}
\mathscr{B}_{fin}\times\mathscr{B}_{fin}\xrightarrow{\quad\beta\quad}\mathbb{R};
\end{equation}
i.e., it is assumed that  $\forall n\in\mathbb{N}$, $\forall\left\{ c_{i}\right\} _{1}^{n}$,
$\left\{ A_{i}\right\} _{1}^{n}$, $c_{i}\in\mathbb{R}$, $A_{i}\in\mathscr{B}_{fin}$,
we have 
\begin{equation}
\sum_{i}\sum_{j}c_{i}c_{j}\beta\left(A_{i},A_{j}\right)\geq0.\label{eq:GI3}
\end{equation}
Then let $X=X^{\left(\beta\right)}$ be the Gaussian process with
\begin{equation}
\left\{ \begin{split} & \mathbb{E}\left(X_{A}\right)=0,\;\text{and}\\
 & \mathbb{E}\left(X_{A}X_{B}\right)=\beta\left(A,B\right),\;\text{for }\text{\ensuremath{\forall A,B\in\mathscr{B}_{fin}}}.
\end{split}
\right.\label{eq:GI25}
\end{equation}

\begin{thm}
\label{thm:G7}Let $\left(M,\mathscr{B},\mu\right)$ be as above,
and let $\beta$ be a corresponding p.d. function, i.e., we have (\ref{eq:GI1})--(\ref{eq:GI3})
satisfied. 

Now suppose there is a Hilbert space $\mathscr{H}$ such that the
conditions in \thmref{T1} are satisfied. 

Then the Gaussian process $X=X^{\left(\beta\right)}$ admits an It\^{o}-integral
representation: Let $X^{\left(\mu\right)}$ denote the Gaussian field
from \propref{G1} and \corref{G2}. Then there is a function $l$,
as follows: 
\begin{equation}
\left\{ \begin{split}\mathscr{B}_{fin} & \xrightarrow{\quad l\quad}L^{2}\left(M,\mu\right)\\
\overset{\rotatebox{90}{\text{\ensuremath{\in}}}}{A} & \xmapsto{\quad\phantom{l}\quad}\overset{\rotatebox{90}{\text{\ensuremath{\in}}}}{l_{A}}
\end{split}
\right.
\end{equation}
such that 
\begin{equation}
X_{A}=\int_{M}l_{A}\left(x\right)dX_{x}^{\left(\mu\right)},\quad\forall A\in\mathscr{B}_{fin};\label{eq:GI27}
\end{equation}
where (\ref{eq:GI27}) is the It\^{o}-integral from \corref{G2}.
\end{thm}

We shall first need a lemma which may be of independent interest. 
\begin{lem}
\label{lem:G8}With the conditions on $\left(\beta,\mu\right)$ as
in the statement of \thmref{T1} and \thmref{G7}, we get existence
of an $L^{2}\left(\mu\right)$-factorization for the initially given
p.d. function $\beta$ \textup{(}see (\ref{eq:GI1})--(\ref{eq:GI3})\textup{)}.
Specifically, $\beta$ admits a representation: 
\begin{equation}
\beta\left(A,B\right)=\int_{M}l_{A}\left(x\right)l_{B}\left(x\right)d\mu\left(x\right),\quad\forall A,B\in\mathscr{B}_{fin}\label{eq:GI28}
\end{equation}
with $l_{A}\in L^{2}\left(\mu\right)$, $\forall A\in\mathscr{B}_{fin}$. 
\end{lem}

\begin{proof}[Proof of the lemma]
 An application of \thmref{T1} yields a closed linear operator $T$
from $L^{2}\left(\mu\right)$ into $\mathscr{H}$, having $\mathscr{D}_{fin}\left(\mu\right)\subset L^{2}\left(\mu\right)$
as dense domain. Moreover, we have: 
\begin{eqnarray*}
\beta\left(A,B\right) & \underset{\text{by \ensuremath{\left(\ref{eq:S4}\right)}}}{=} & \left\langle \chi_{A},\chi_{B}\right\rangle _{\mathscr{H}}\\
 & \underset{\text{by \ensuremath{\left(\ref{eq:T3}\right)}}}{=} & \left\langle T\left(\chi_{A}\right),T\left(\chi_{B}\right)\right\rangle _{\mathscr{H}}\\
 & = & \left\langle T^{*}T\chi_{A},\chi_{B}\right\rangle _{L^{2}\left(\mu\right)}\\
 & \underset{{\scriptscriptstyle \stackrel{\text{\text{since \ensuremath{T^{*}T} is}}}{\text{selfadjoint}}}}{=} & \left\langle \left(\left(T^{*}T\right)^{\frac{1}{2}}\right)^{2}\chi_{A},\chi_{B}\right\rangle _{L^{2}\left(\mu\right)}\\
 & = & \left\langle \left(T^{*}T\right)^{\frac{1}{2}}\chi_{A},\left(T^{*}T\right)^{\frac{1}{2}}\chi_{B}\right\rangle _{L^{2}\left(\mu\right)}.
\end{eqnarray*}
Now setting, 
\begin{equation}
l_{A}:=\left(T^{*}T\right)^{\frac{1}{2}}\chi_{A},\quad A\in\mathscr{B}_{fin},\label{eq:GI29}
\end{equation}
the desired conclusion (\ref{eq:GI28}) follows.
\end{proof}
\begin{proof}[Proof of \thmref{G7}]
 Let $\left(\beta,\mu\right)$ be as in the statement of \thmref{G7},
and let $\left\{ l_{A}\right\} _{A\in\mathscr{B}_{fin}}$ be the $L^{2}\left(\mu\right)$-function
in (\ref{eq:GI29}). We see that the factorization (\ref{eq:GI28})
is valid. 

Hence, by \corref{G2}, the corresponding It\^{o}-integral (\ref{eq:GI27})
is well defined; and the resulting Gaussian process $X_{A}:=\int_{M}l_{A}\left(x\right)dX_{x}^{\left(\mu\right)}$
is a Gaussian field with $\mathbb{E}\left(X_{A}\right)=0$. Hence
we only need to verify the convariance condition in (\ref{eq:GI25})
above: 

Let $A,B\in\mathscr{B}_{fin}$, and compute: 
\begin{eqnarray*}
\mathbb{E}\left(X_{A}X_{B}\right) & = & \mathbb{E}\left[\left(\int_{M}l_{A}\left(x\right)dX_{x}^{\left(\mu\right)}\right)\left(\int_{M}l_{B}\left(x\right)dX_{x}^{\left(\mu\right)}\right)\right]\\
 & \underset{{\scriptscriptstyle \text{by Cor. \ref{cor:G3}}}}{=} & \int_{M}l_{A}\left(x\right)l_{B}\left(x\right)d\mu\left(x\right)\quad\left(\mu=QV\left(X^{\left(\mu\right)}\right)\right)\\
 & \underset{{\scriptscriptstyle {\scriptscriptstyle \stackrel{\text{by Lem. \ref{lem:G8},}}{\text{see \ensuremath{\left(\ref{eq:GI28}\right)}}}}}}{=} & \beta\left(A,B\right);
\end{eqnarray*}
and the proof is completed. 
\end{proof}
\begin{rem}[fractional Brownian motion]
\label{rem:GI9}As an application of \thmref{G7}, consider the case
of $\left(\mathbb{R},\mathscr{B},\lambda_{1}\right)$ (so $\mu=\lambda_{1}$),
i.e., standard Lebesgue measure on $\mathbb{R}$, with $\mathscr{B}$
denoting the standard Borel-sigma-algebra. We shall discuss fractional
Brownian motion with Hurst parameter $H$ (see \cite{MR0242239,MR672910,MR2123205,MR2178502,MR2793121,MR2966130}). 

Recall, on $[0,\infty)$, fractional Brownian motion $\{X_{t}^{\left(H\right)}\}_{t\in[0,\infty)}$,
$0<H<1$, fixed, may be normalized as follows: $X_{0}^{\left(H\right)}=0$,
\begin{align}
\mathbb{E}\left(X_{t}^{\left(H\right)}\right) & =0,\quad\text{and}\\
\mathbb{E}\left(X_{s}^{\left(H\right)}X_{t}^{\left(H\right)}\right) & =\frac{1}{2}\left(s^{2H}+t^{2H}-\left|s-t\right|^{2H}\right),\quad\forall s,t\in[0,\infty).\label{eq:GI31}
\end{align}
The corresponding process induced by $\mathscr{B}_{fin}$ is 
\begin{equation}
X_{\left[0,t\right]}^{\left(H\right)}:=X_{t}^{\left(H\right)};\label{eq:GI32}
\end{equation}
and we shall adapt (\ref{eq:GI32}) as an identification. The following
spectral representation is known: Set, for $\lambda\in\mathbb{R}$,
\begin{equation}
d\mu^{\left(H\right)}\left(\lambda\right)=\frac{\sin\left(\pi H\right)\Gamma\left(1+2H\right)}{2\pi}\left|\lambda\right|^{1-2H}d\lambda;\label{eq:G33}
\end{equation}
then 
\begin{equation}
\mathbb{E}\left(X_{s}^{\left(H\right)}X_{t}^{\left(H\right)}\right)=\int_{\mathbb{R}}\frac{\left(e^{i\lambda s}-1\right)\left(e^{-i\lambda t}-1\right)}{\lambda^{2}}d\mu^{\left(H\right)}\left(\lambda\right).\label{eq:G34}
\end{equation}

A choice of factorization for the kernel $K^{\left(H\right)}\left(s,t\right)$
in (\ref{eq:GI31}) is then as follows: 
\begin{align*}
K^{\left(H\right)}\left(s,t\right) & =\frac{1}{2}\left(s^{2H}+t^{2H}-\left|s-t\right|^{2H}\right)\\
 & =\int_{\mathbb{R}}l_{s}\left(x\right)l_{t}\left(x\right)dx\quad\left(s,t\in[0,\infty)\right)
\end{align*}
with 
\begin{align}
l_{t}\left(x\right) & =\frac{1}{\Gamma\left(H+\frac{1}{2}\right)}\Big(\chi_{(-\infty,0]}\left(x\right)\left(\left(t-x\right)^{H-\frac{1}{2}}-\left(-x\right)^{H-\frac{1}{2}}\right)\nonumber \\
 & \qquad+\chi_{\left[0,t\right]}\left(x\right)\left(t-x\right)^{H-\frac{1}{2}}\Big),\quad x\in\mathbb{R},\:t\in[0,\infty).\label{eq:G35}
\end{align}
\end{rem}

\subsection{Application to fractional Brownian motion}

\hspace*{\fill}\\Fix $H$, $0<H<1$, the Hurst parameter, and let
$\{X_{t}^{\left(H\right)}\}_{t\in[0,\infty)}$ be fractional Brownian
motion (fBM), see \remref{GI9}. Then the special case $H=\frac{1}{2}$
corresponds to standard Brownian motion (BM). we shall write $X_{t}^{(\nicefrac{1}{2})}=W_{t}$;
where ``$W$'' is for Wiener. Now $\left\{ W_{t}\right\} _{t\in[0,\infty)}$
is a martingale; and standard Brownian motion has \emph{independent
increments}, by contrast to the case when $H\neq\frac{1}{2}$, i.e.,
fBM. 
\begin{flushleft}
\textbf{(i) It\^{o}-integral representation for $X_{t}^{\left(H\right)}$
when $H\neq\frac{1}{2}$.}
\par\end{flushleft}

We now combine \thmref{G7}, (\ref{eq:G35}) and (\ref{eq:GI27})
to conclude that $X_{t}^{\left(H\right)}$ has the following It\^{o}-integral
representation: 

Let $\{l_{t}^{\left(H\right)}\}_{t\in[0,\infty)}$ be the integral
kernel from (\ref{eq:G35}). Note, it depends on the value of $H$,
but we shall fix $H$, $H\neq\frac{1}{2}$. Then 
\begin{equation}
X_{t}^{\left(H\right)}=\int_{\mathbb{R}}l_{t}^{\left(H\right)}\left(x\right)dW_{x};\label{eq:GI36}
\end{equation}
where $\text{RHS}_{\left(\ref{eq:GI36}\right)}$ is the It\^{o}-integral
introduced in \corref{G2} in the more general setting of $X^{\left(\mu\right)}$.
Here, $\mu=\lambda_{1}=dx$ is standard Lebesgue measure; and $QV\left(W_{x}\right)=dx$;
see \corref{G3}.
\begin{flushleft}
\textbf{(ii) Filtrations.}
\par\end{flushleft}

Returning to the probability space $\left(\Omega,\mathscr{C}\right)$
for $\left\{ W_{t}\right\} _{t\in[0,\infty)}$; see \propref{G1},
and let $\mathscr{B}$ be the standard Borel $\sigma$-algebra of
subsets of $\mathbb{R}$. For $A\in\mathscr{B}$, we denote by $\mathscr{F}\left(A\right):=$
the sub $\sigma$-algebra of the cylinder $\sigma$-algebra in $\Omega$
(see (\ref{eq:G4})) generated by the random variables $W_{B}$, as
$B$ in $\mathscr{B}$ varies over subsets $B\subseteq A$. Let $l_{t}^{\left(\pm\right)}\left(x\right)$
denote the two separate terms on $\text{RHS}_{\left(\ref{eq:G35}\right)}$,
i.e., 
\[
l_{t}^{\left(-\right)}\left(x\right)=\chi_{(-\infty,0]}\left(x\right)\left(\left(t-x\right)^{H-\frac{1}{2}}-\left(-x\right)^{H-\frac{1}{2}}\right)\Big/\Gamma\left(H+\tfrac{1}{2}\right)
\]
and 
\[
l_{t}^{\left(+\right)}\left(x\right)=\chi_{\left[0,t\right]}\left(x\right)\left(t-x\right)^{H-\frac{1}{2}}\Big/\Gamma\left(H+\tfrac{1}{2}\right).
\]
Then there are two components (of fractional Brownian motion):
\[
X_{t}^{\left(-\right)}=\int_{-\infty}^{0}l_{t}^{\left(-\right)}\left(x\right)dW_{x},
\]
and 
\[
X_{t}^{\left(+\right)}=\int_{0}^{t}l_{t}^{\left(+\right)}\left(x\right)dW_{x};
\]
where $H\neq\frac{1}{2}$ is fixed; (supposed in the notation.)

The two processes $(X_{t}^{\left(-\right)})$ and $(X_{t}^{\left(+\right)})$
are independent, and 
\begin{equation}
X_{t}=X_{t}^{\left(H\right)}=X_{t}^{\left(-\right)}\oplus X_{t}^{\left(+\right)}
\end{equation}
with 
\begin{equation}
\mathbb{E}\left(X_{s}X_{t}\right)=\mathbb{E}\left(X_{s}^{\left(-\right)}X_{t}^{\left(-\right)}\right)+\mathbb{E}\left(X_{s}^{\left(+\right)}X_{t}^{\left(+\right)}\right),\quad\forall s,t\in[0,\infty).
\end{equation}

These processes $(X_{t}^{\left(\pm\right)})$ result from the initial
fBM $X_{t}$ (\ref{eq:GI36}) itself, as conditional Gaussian processes
as follows: 
\begin{equation}
\mathbb{E}\left(X_{t}^{\left(+\right)}\mid\mathscr{F}\left((-\infty,0]\right)\right)=0;\label{eq:G39}
\end{equation}
and 
\begin{equation}
\mathbb{E}\left(X_{t}\mid\mathscr{F}\left((-\infty,0]\right)\right)=X_{t}^{\left(-\right)}\label{eq:G38}
\end{equation}
and 
\begin{equation}
\mathbb{E}\left(X_{t}\mid\mathscr{F}\left(\left[0,t\right]\right)\right)=X_{t}^{\left(+\right)}.\label{eq:G41}
\end{equation}
So $X_{t}^{\left(-\right)}$ in (\ref{eq:G38}) is the backward process,
while $X_{t}^{\left(+\right)}$ is the corresponding forward process. 
\begin{cor}
Fix $H$ (Hurst parameter) as above, and consider the fractional Brownian
motion $X_{t}^{\left(H\right)}$, and its forward part $X_{t}^{\left(+\right)}:=(X_{t}^{\left(H\right)})^{+}$
given in (\ref{eq:G41}). Then $(X_{t}^{\left(H\right)})^{+}$ is
a \uline{semimartingale}, i.e., if $0<s<t$, then 
\begin{equation}
\mathbb{E}\left(X_{t}^{\left(+\right)}\mid\mathscr{F}\left(\left[0,s\right]\right)\right)=X_{s}^{\left(+\right)}.\label{eq:GS42}
\end{equation}
\end{cor}

\begin{proof}
\begin{eqnarray*}
\text{LHS}_{\left(\ref{eq:GS42}\right)} & \underset{{\scriptscriptstyle \text{by \ensuremath{\left(\ref{eq:G41}\right)}}}}{=} & \mathbb{E}\left(\mathbb{E}\left(X_{t}^{\left(H\right)}\mid\mathscr{F}\left(\left[0,t\right]\right)\right)\mid\mathscr{F}\left(\left[0,s\right]\right)\right)\\
 & = & \mathbb{E}\left(X_{t}^{\left(H\right)}\mid\mathscr{F}\left(\left[0,s\right]\right)\right)\\
 & = & \mathbb{E}\left(\left(X_{t}^{\left(H\right)}-X_{s}^{\left(H\right)}\right)+X_{s}^{\left(H\right)}\mid\mathscr{F}\left(\left[0,s\right]\right)\right)\\
 & \underset{{\scriptscriptstyle \text{by \ensuremath{\left(\ref{eq:G39}\right)}}}}{=} & \mathbb{E}\left(X_{s}^{\left(H\right)}\mid\mathscr{F}\left(\left[0,s\right]\right)\right)\underset{{\scriptscriptstyle \text{by \ensuremath{\left(\ref{eq:G41}\right)}}}}{=}X_{s}^{\left(+\right)}.
\end{eqnarray*}
\end{proof}
We stress that the proofs of these properties of fBM, (with $H\neq\frac{1}{2}$)
follow essentially from our conclusions in \remref{GI9}, as well
as Corollaries \ref{cor:G2} and \ref{cor:G3}.
\begin{flushleft}
\textbf{The spectral representation.}
\par\end{flushleft}

The formula (\ref{eq:G34}) is a \emph{spectral representation} in
following sense: The choice of $d\mu^{\left(H\right)}$ in (\ref{eq:G33})
yields the following generalized Paley-Wiener space (compare (\ref{eq:I7})--(\ref{eq:I8})
above):

Let $\mathscr{H}(\mu^{\left(H\right)})$ denote the Hilbert space
of functions $f$ on $\mathbb{R}$ such that the Fourier transform
$\widehat{f}$ is well defined and is in $L^{2}(\mu^{\left(H\right)})$.
Then set
\begin{equation}
\left\Vert f\right\Vert _{\mathscr{H}(\mu^{\left(H\right)})}^{2}=\Vert\widehat{f}\Vert_{L^{2}(\mu^{\left(H\right)})}^{2}=\int_{\mathbb{R}}|\widehat{f}\left(\lambda\right)|^{2}d\mu^{\left(H\right)}\left(\lambda\right).\label{eq:G40}
\end{equation}
For $f\in\mathscr{H}(\mu^{\left(H\right)})$, consider the It\^{o}-integral,
\begin{equation}
X^{\left(H\right)}\left(f\right):=\int f\left(t\right)dX_{t}^{\left(H\right)}.
\end{equation}
Then it follows from (\ref{eq:G34}), and Theorems \ref{thm:T1} and
\ref{thm:G7} that 
\begin{equation}
\mathbb{E}\left(\left|X^{\left(H\right)}\left(f\right)\right|^{2}\right)=\left\Vert f\right\Vert _{\mathscr{H}(\mu^{\left(H\right)})}^{2}.
\end{equation}
In particular, 
\begin{equation}
\mathbb{E}\left(\left|X^{\left(H\right)}\left(f\left(\cdot+t\right)\right)\right|^{2}\right)=\mathbb{E}\left(\left|X^{\left(H\right)}\left(f\right)\right|^{2}\right).
\end{equation}
This follows since the RHS in (\ref{eq:G40}) is translation invariant,
i.e., we have:
\begin{equation}
\widehat{f\left(\cdot+t\right)}\left(\lambda\right)=e^{i\lambda t}\widehat{f}\left(\lambda\right).
\end{equation}

\subsection{A Karhunen-Lo\`{e}ve representation}

\hspace*{\fill}\\The Karhunen-Lo\`{e}ve (KL) theorem is usually
stated for the special case of positive definite kernels $K$ which
are also continuous (typically on a bounded interval), so called Mercer-kernels.
The starting point is then an application of the spectral theorem
to the corresponding selfadjoint integral operators, $T_{K}$ in $L^{2}$
of the interval. Mercer\textquoteright s theorem states that if $K$
is Mercer, then the integral operator $T_{K}$ is trace-class. A Karhunen-Lo\`{e}ve
representation for a stochastic process (with specified covariance
kernel $K$) is a generalized infinite linear combination, or orthogonal
expansion, for the random process, analogous to a Fourier series representation
for (deterministic) functions on a bounded interval; see e.g., \cite{MR0008270,MR2268393}.
The KL representation we give below is much more general, and it applies
to the most general positive definite kernel, and makes essential
use of our RKHS theorem (\corref{G7} below). In our KL-theorem, we
also make precise the random i.i.d $N(0,1)$-terms inside the KL-expansion;
see (\ref{eq:G36}).
\begin{cor}
\label{cor:G7}Let $\left(M,\mathscr{B},\mu\right)$ be a $\sigma$-finite
measure space, and let $\{X_{A}^{\left(\mu\right)}\}_{A\in\mathscr{B}_{fin}}$
be the associated Gaussian field (see \propref{G1} and \corref{G2}.)
Let $\left\{ \varphi_{k}\right\} _{k\in\mathbb{N}}$ be an orthonormal
basis (ONB) in $L^{2}\left(\mu\right)$, and set 
\begin{equation}
Z_{k}:=X_{\varphi_{k}}^{\left(\mu\right)}=\int_{M}\varphi_{k}\,dX^{\left(\mu\right)}.\label{eq:G36}
\end{equation}
\begin{enumerate}
\item Then $\left\{ Z_{k}\right\} _{k\in\mathbb{N}}$ is an i.i.d. $N\left(0,1\right)$-system
(i.e., a system of independent, identically distributed standard Gaussians.)
\item Moreover, $X^{\left(\mu\right)}$ admits the following Karhunen-Lo\`{e}ve
representation ($A\in\mathscr{B}_{fin}$): 
\begin{equation}
X_{A}^{\left(\mu\right)}\left(\cdot\right)=\sum_{k\in\mathbb{N}}\left(\int_{A}\varphi_{k}\,d\mu\right)Z_{k}\left(\cdot\right),\label{eq:G23}
\end{equation}
and, more generally, for $\psi\in L^{2}\left(\mu\right)$, 
\begin{equation}
X_{\psi}^{\left(\mu\right)}\left(\cdot\right)=\sum_{k\in\mathbb{N}}\left\langle \psi,\varphi_{k}\right\rangle _{L^{2}\left(\mu\right)}Z_{k}\left(\cdot\right).
\end{equation}
\item In particular, $X^{\left(\mu\right)}$ admits a realization on the
infinite product space $\Omega=\mathbb{R}^{\mathbb{N}}$, equipped
with the usual cylinder $\sigma$-algebra, and the infinite-product
measure
\begin{equation}
\mathbb{P}:=\vartimes_{\mathbb{N}}g_{1}=g_{1}\times g_{1}\times\cdots,\label{eq:G25}
\end{equation}
where $g_{1}\left(x\right)=\frac{1}{\sqrt{2\pi}}e^{-x^{2}/2}=$ the
$N\left(0,1\right)$-distribution. (We compute the expectation $\mathbb{E}$
with respect to $\mathbb{P}$, the infinite product measure $\mathbb{P}$
in (\ref{eq:G25}))
\end{enumerate}
\end{cor}

\begin{proof}[Proof sketch]
 When the system $\left\{ Z_{k}\right\} _{k\in\mathbb{N}}$ is specified
as in (\ref{eq:G36}), it follows from standard Gaussian theory (see
e.g., \cite{MR2362796,MR2446590,MR2670567,MR2793121,MR2966130,MR3402823,MR3687240}
and the papers cited there) that it is an i.i.d. $N\left(0,1\right)$-system. 

For $A,B\in\mathscr{B}_{fin}$, 
\begin{eqnarray*}
\mathbb{E}\left(X_{A}^{\left(\mu\right)}X_{B}^{\left(\mu\right)}\right) & = & \sum_{k\in\mathbb{N}}\sum_{l\in\mathbb{N}}\int_{A}\varphi_{k}\,d\mu\int_{B}\varphi_{l}\,d\mu\,\underset{{\scriptscriptstyle \delta_{k,l}}}{\underbrace{\mathbb{E}\left(Z_{k}Z_{l}\right)}}\\
 & = & \sum_{k\in\mathbb{N}}\int_{A}\varphi_{k}\,d\mu\int_{B}\varphi_{k}\,d\mu\\
 & \underset{{\scriptscriptstyle \text{by Parseval}}}{=} & \left\langle \chi_{A},\chi_{B}\right\rangle _{L^{2}\left(\mu\right)}=\mu\left(A\cap B\right).
\end{eqnarray*}
Since the representation in (\ref{eq:G23}) yields a Gaussian process
with mean zero, it is determined by its covariance kernel, and the
result follows.
\end{proof}
\begin{rem}
In this section, we have addressed some questions that are naturally
implied by our present setting, but we wish to stress that there is
a vast literature in the general area of the subject, and dealing
with a variety of different important issues for Gaussian fields.
Below we cite a few papers, and readers may also want to consult papers
cited there: \cite{MR2677883,MR2805533,MR3266271,MR3803148,MR1952822,MR2382071,MR3501849,MR2322706}.
\end{rem}

\section{Measures on $\left(I,\mathscr{B}\right)$ when $I$ is an interval}

We consider the spaces consisting of the measure spaces when $I$
is an interval (taking $I=\left[0,1\right]$ for specificity); and
where $\mathscr{B}$ is the standard Borel $\sigma$-algebra of subsets
in $I$. 

In this case, our results above, especially \corref{T2}, take the
following form: 
\begin{thm}
Let $\mu$ be a $\sigma$-finite measure on $\left(M,\mathscr{B}\right)$,
and $\beta_{\mu}\left(A,B\right)=\mu\left(A\cap B\right)$ the p.d.
function from (\ref{eq:T8}). Let $\mathscr{H}\left(\beta_{\mu}\right)$
be the corresponding RKHS. 
\begin{enumerate}
\item Then $\mathscr{H}\left(\beta_{\mu}\right)$ consists of all functions
$F$ on $\left[0,1\right]$, such that $F\left(0\right)=0$, and 
\begin{equation}
\sup_{0\leq a<b\leq1}\frac{\left|F\left(b\right)-F\left(a\right)\right|}{\mu\left(\left[a,b\right]\right)}<\infty,\label{eq:M1}
\end{equation}
supremum over all intervals contained in $\left[0,1\right]$. 
\item If $dF/d\mu$ denotes the Radon-Nikodym derivative corresponding to
(\ref{eq:M1}), then the $\mathscr{H}\left(\beta_{\mu}\right)$-norm
is as follows:

First, $dF/d\mu\in L^{2}\left(\mu\right)$, and 
\begin{equation}
\left\Vert F\right\Vert _{\mathscr{H}\left(\beta_{\mu}\right)}^{2}=\int_{0}^{1}\left|\frac{dF}{d\mu}\right|^{2}d\mu.\label{eq:M2}
\end{equation}

\end{enumerate}
\end{thm}

\begin{proof}[Proof sketch]
 The idea is essentially contained in the considerations above from
\secref{SA}. Indeed, if $F$ is as specified in (\ref{eq:M1}) \&
(\ref{eq:M2}), set for all $A\in\mathscr{B}$, 
\begin{equation}
\mu_{F}\left(A\right)=\int_{A}\frac{dF}{d\mu}d\mu.\label{eq:M3}
\end{equation}
Then the Radon-Nikodym derivative $d\mu_{F}/d\mu$ in (\ref{eq:T7})
satisfies $d\mu_{F}/d\mu=dF/d\mu$ (see (\ref{eq:M2})--(\ref{eq:M3})). 

Moreover, 
\begin{equation}
L^{2}\left(\mu\right)\ni\varphi\longmapsto\underset{F_{\varphi}\left(\cdot\right)\in\mathscr{H}\left(\beta_{\mu}\right)}{\underbrace{\int_{A}\varphi\,d\mu=F_{\varphi}\left(A\right)}}
\end{equation}
defines an isometry, mapping onto $\mathscr{H}\left(\beta_{\mu}\right)$.
\end{proof}
\begin{example}[Cantor measures]
 If, for example, $\mu=\mu_{3}$ is the middle-third Cantor measure,
then the Devil's Staircase function (see \figref{cum}) is 
\begin{equation}
F\left(x\right)=\mu_{3}\left(\left[0,x\right]\right).
\end{equation}
It is in $\mathscr{H}\left(\beta_{\mu}\right)$, and 
\begin{equation}
\frac{dF}{d\mu_{3}}=\chi_{\left[0,1\right]}.\label{eq:M6}
\end{equation}
Note that it is important that the Radon-Nikodym derivative in (\ref{eq:M6})
is with respect to the Cantor measure $\mu_{3}$. If, for example,
$\lambda$ denotes the Lebesgue measure on $\left[0,1\right]$, then
$dF/d\lambda=0$. 

For graphical illustration of these functions, see Figures \ref{fig:DS}--\ref{fig:cum}
below. 
\end{example}

\begin{figure}[H]
\includegraphics[width=0.4\textwidth]{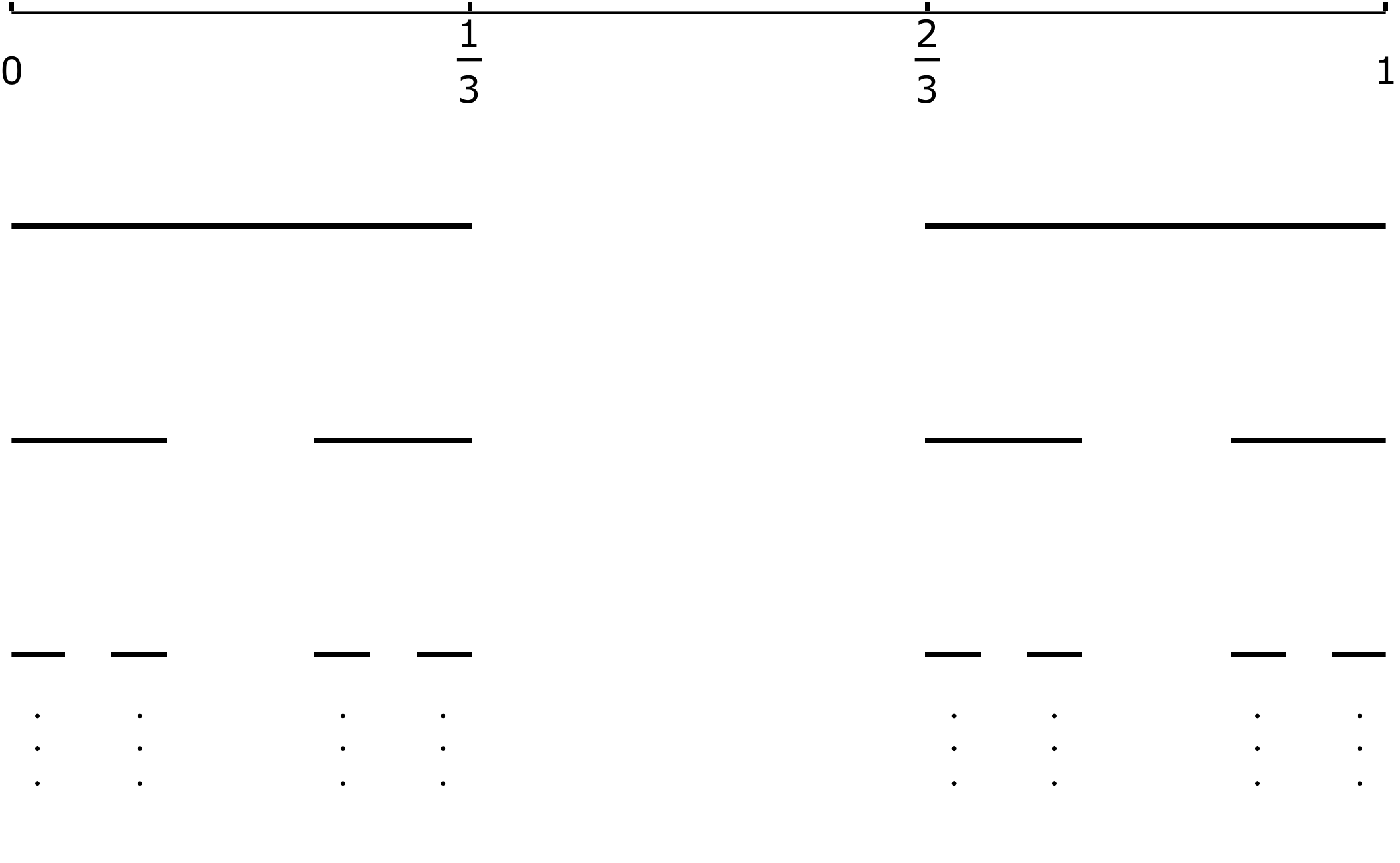}

\caption{\label{fig:DS}The middle-third Cantor set.}
\end{figure}

\begin{flushleft}
\textbf{Time-change}
\par\end{flushleft}

While there is earlier work in the literature, dealing with time-change
in Gaussian processes, see e.g., \cite{MR2484103,MR3363697}; our
aim here is to illustrate the use of our results in Sections \ref{sec:SA}
and \ref{sec:GF} as they apply to the change of the time-variable
in a Gaussian process. To make our point, we have found it sufficient
to derive the relevant properties for time-change for time in a half-line.

\begin{figure}[H]
\begin{tabular}{>{\raggedright}p{0.45\textwidth}>{\raggedright}p{0.45\textwidth}}
\includegraphics[width=0.3\textwidth]{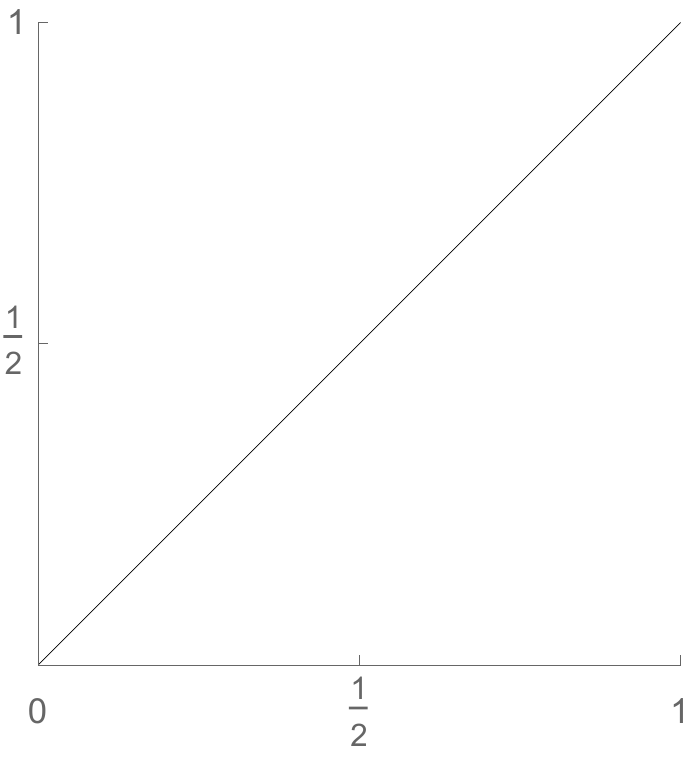} & \includegraphics[width=0.3\textwidth]{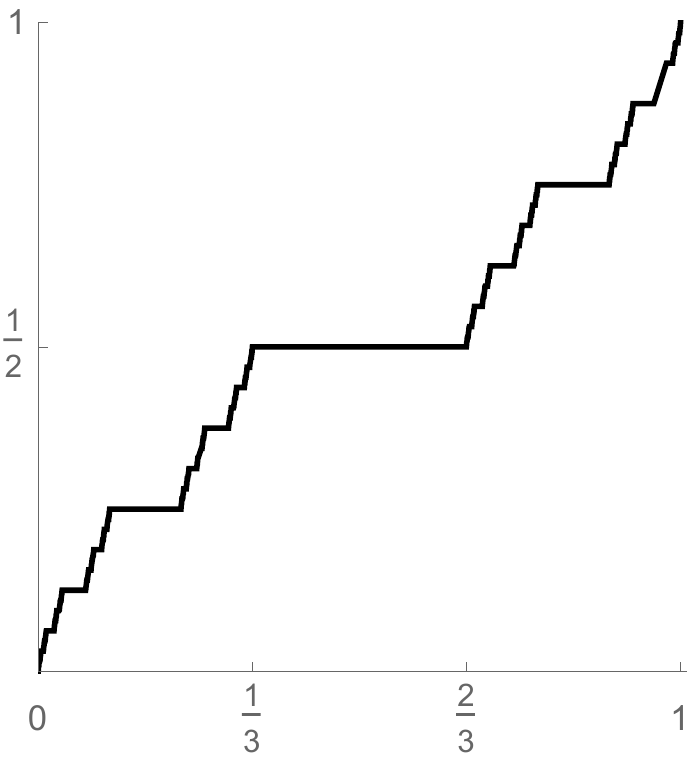}\tabularnewline
{\footnotesize{} $F_{\lambda}\left(x\right)=\lambda\left(\left[0,x\right]\right)$;
points of increase = the support of the normalized $\lambda$, so
the interval $[0,1]$. } & {\footnotesize{} $F_{\nicefrac{1}{3}}\left(x\right)=\mu_{3}\left(\left[0,x\right]\right)$;
points of increase = the support of $\mu_{3}$, so the middle third
Cantor set $C_{\nicefrac{1}{3}}$ (the Devil's staircase).}\tabularnewline
\end{tabular}\vspace{1em}

\caption{\label{fig:cum}The two \emph{cumulative distributions}, with support
sets $[0,1]$ and $C_{\nicefrac{1}{3}}$. }
\end{figure}

\begin{prop}
\label{prop:M3}Let $J=[0,\infty)$ denote the positive half-line,
and let $\left\{ B_{t}\right\} _{t\in J}$ be the standard Brownian
motion, i.e., $B_{t}\sim N\left(0,t\right)$, and 
\begin{equation}
\mathbb{E}\left(B_{s}B_{t}\right)=s\wedge t,\quad\forall s,t\in J\label{eq:M7}
\end{equation}
where $s\wedge t=\min\left(s,t\right)$. Let $h:J\rightarrow J$ be
a monotone (increasing) function such that $h\left(0\right)=0$, and
set $X=X^{\left(h\right)}$ given by 
\begin{equation}
X_{t}:=B_{h\left(t\right)},\quad t\in J.\label{eq:M8}
\end{equation}
\begin{enumerate}
\item Then $X_{t}$ is the Gaussian process determined by the following
induced covariance kernel:
\begin{equation}
\mathbb{E}\left(X_{s}X_{t}\right)=h\left(s\wedge t\right)
\end{equation}
 
\item \label{enu:M32}The \uline{quadratic variation measure} for $\{X_{t}^{\left(h\right)}\}_{t\in J}$
is 
\begin{equation}
d\mu\left(t\right)=h'\left(t\right)dt,\label{eq:M10}
\end{equation}
where $dt$ is the usual Lebesgue measure on $J$. 

(Recall that, since $h\left(s\right)\leq h\left(t\right)$ for all
$s,t$, $s\leq t$; it follows, by Lebesgue's theorem, that $h$ is
differentiable almost everywhere on $J$ with respect to $dt$.)

\end{enumerate}
\end{prop}

\begin{rem}
Note that, if $h\left(t\right)=t^{2}$, then $\mathbb{E}\left(X_{s}X_{t}\right)=\left(s\wedge t\right)^{2}$;
see \figref{time} for an illustration.
\end{rem}

\begin{proof}
Since $h$ is monotone (increasing) and $h\left(0\right)=0$, we get
\begin{equation}
h\left(s\right)\wedge h\left(t\right)=h\left(s\wedge t\right),\quad\forall s,t\in J\label{eq:M11}
\end{equation}
and so the covariance kernel satisfies:
\begin{eqnarray*}
\mathbb{E}\left(X_{s}X_{t}\right) & = & \mathbb{E}\left(B_{h\left(s\right)}B_{h\left(t\right)}\right)\\
 & \underset{{\scriptscriptstyle \text{by \ensuremath{\left(\ref{eq:M7}\right)}}}}{=} & h\left(s\right)\wedge h\left(t\right)\underset{{\scriptscriptstyle \text{by \ensuremath{\left(\ref{eq:M11}\right)}}}}{=}h\left(s\wedge t\right)\\
 & = & \int_{0}^{s\wedge t}h'\left(x\right)dx=\mu\left(s\wedge t\right)\\
 & = & \mu\left(\left[0,s\right]\cap\left[0,t\right]\right)
\end{eqnarray*}
where $\mu$ is the measure given in (\ref{eq:M10}).

It now follows from \corref{G3} that then $\mu$ is indeed the quadratic
variation measure for $\{X_{t}^{\left(h\right)}\}_{t\in J}$, as asserted. 
\end{proof}
\begin{figure}[H]
\subfloat[$B_{t}$]{\includegraphics[width=0.4\textwidth]{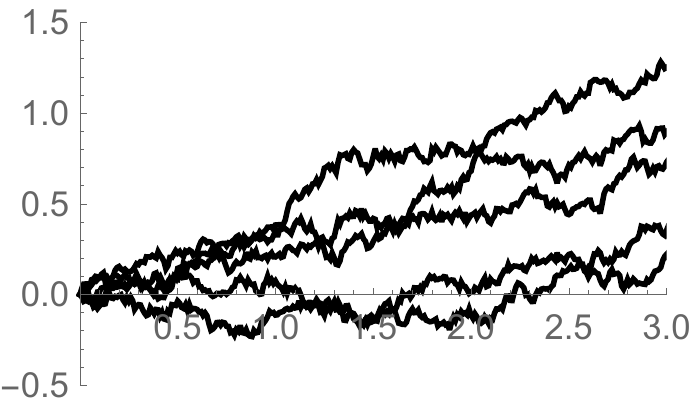}

}\hspace{2cm}\subfloat[$B_{t^{2}}$]{\includegraphics[width=0.4\textwidth]{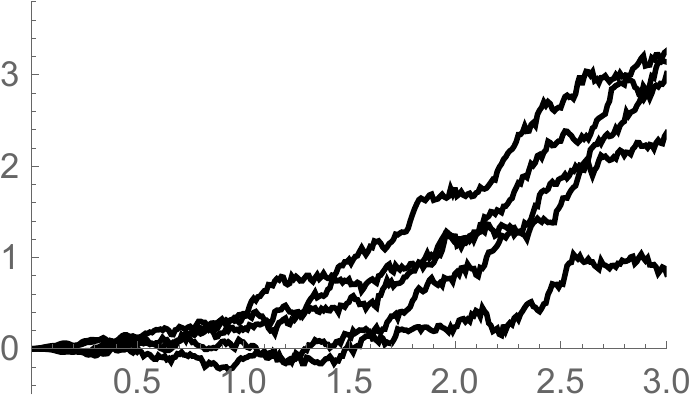}

}

\caption{\label{fig:time}Time-change of Brownian motion}

\end{figure}

\begin{cor}
\label{cor:M5}Let $h:J\rightarrow J$ be as in \propref{M3}, i.e.,
$h\left(0\right)=0$, $h\left(s\right)\leq h\left(t\right)$, for
$s\leq t$; and, as in (\ref{eq:M8}), consider:
\begin{equation}
X_{t}=B_{h\left(t\right)},\quad t\in J.
\end{equation}
Let $f:\mathbb{R}\rightarrow\mathbb{R}$ be given, assumed twice differentiable.
Then the It\^{o}-integral formula for $f\left(X_{t}\right)$ is as
follows: For $t>0$, we have:
\begin{equation}
f\left(X_{t}\right)=\int_{0}^{t}f'\left(X_{s}\right)dX_{s}+\frac{1}{2}\int_{0}^{t}f''\left(X_{s}\right)h'\left(s\right)ds.\label{eq:M13}
\end{equation}
\end{cor}

\begin{proof}
The result is immediate from It\^{o}'s lemma applied to the quadratic
variation term on the right-hand side in (\ref{eq:M13}). 

Recall, we proved in \propref{M3} (\ref{enu:M32}), eq. (\ref{eq:M10})
that the quadratic variation of a Gaussian process with covariance
measure $\mu$ is $\mu$ itself. Hence (\ref{eq:M13}) follows from
a direct application to $d\mu\left(s\right)=h'\left(s\right)ds$,
where $ds$ is standard Lebesgue measure on the interval $J$.
\end{proof}
\begin{cor}
Let $h:J\rightarrow J$, $h\left(0\right)=0$, $h$ monotone be as
specified as \corref{M5}, and let $X_{t}=B_{h\left(t\right)}$ be
the corresponding time-change process. 

Set $d\mu=dh=\left(\text{the Stieltjes measure}\right)=h'\left(t\right)dt$;
see (\ref{eq:M13}). For $\left(t,x\right)\in J\times J$, and $f\in L^{2}\left(\mu\right)$,
let
\begin{equation}
u\left(t,x\right)=\mathbb{E}_{X_{0}=x}\left(f\left(X_{t}\right)\right).\label{eq:M14}
\end{equation}
Then $u$ satisfies the following diffusion equation
\begin{equation}
\frac{\partial}{\partial t}u\left(t,x\right)=\frac{1}{2}h'\left(t\right)\frac{\partial^{2}}{\partial x^{2}}u\left(t,x\right),\label{eq:M15}
\end{equation}
with boundary condition
\[
u\left(t,\cdot\right)\big|_{t=0}=f\left(\cdot\right).
\]
\end{cor}

\begin{proof}
The assertion follows from an application of the conditional expectation
$\mathbb{E}_{X_{0}=x}$ to both sides in (\ref{eq:M13}). Since the
expectation of the first of the two terms on the right-hand side in
(\ref{eq:M13}) vanishes, we get from the definition (\ref{eq:M14})
that:
\[
u\left(t,x\right)=\frac{1}{2}\int_{0}^{t}\mathbb{E}_{X_{0}=x}\left(f''\left(X_{s}\right)\right)h'\left(s\right)ds,
\]
and so 
\[
\frac{\partial}{\partial t}u\left(t,x\right)=\frac{1}{2}h'\left(t\right)\frac{\partial^{2}}{\partial x^{2}}u\left(t,x\right)
\]
as claimed in (\ref{eq:M15}). The remaining conclusions in the corollary
are immediate.
\end{proof}
\begin{rem}
Let $0<H<1$ be fixed, and set 
\begin{equation}
h\left(t\right):=t^{2H},\quad t\in J.
\end{equation}
Then the corresponding process 
\[
X_{t}^{\left(H\right)}:=B_{t^{2H}},\quad t\in J,
\]
is a time-changed process, as discussed in \propref{M3}. We have
\begin{equation}
\mathbb{E}\left((X_{t}^{\left(H\right)})^{2}\right)=t^{2H}.
\end{equation}

Now this is the same variance as the fractional Brownian motion $Y_{t}^{\left(H\right)}$;
but we stress that (when $H$ is fixed, $H\neq1/2$), then the two
Gaussian processes $X_{t}^{\left(H\right)}$ (time-change), and $Y_{t}^{\left(H\right)}$
(fractional Brownian motion with Hurst parameter $H$), are different.
(See \figref{fBM}.) 

\renewcommand{\arraystretch}{1.5}

\begin{figure}
\begin{tabular}{|c|c|}
\hline 
$H=1/3$ & $H=2/3$\tabularnewline
\hline 
\includegraphics[width=0.4\textwidth]{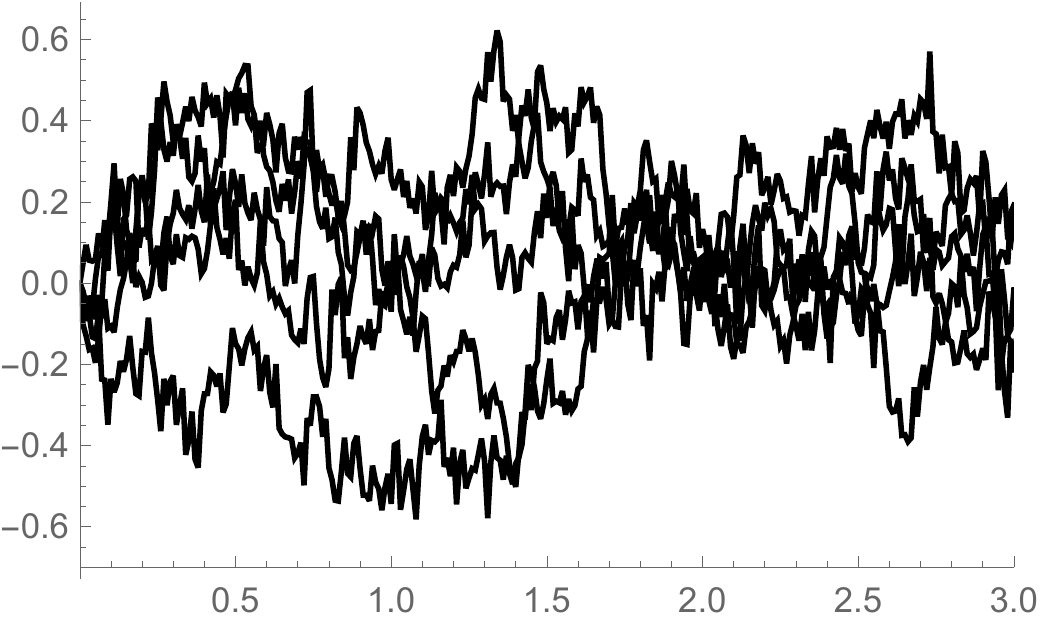} & \includegraphics[width=0.4\textwidth]{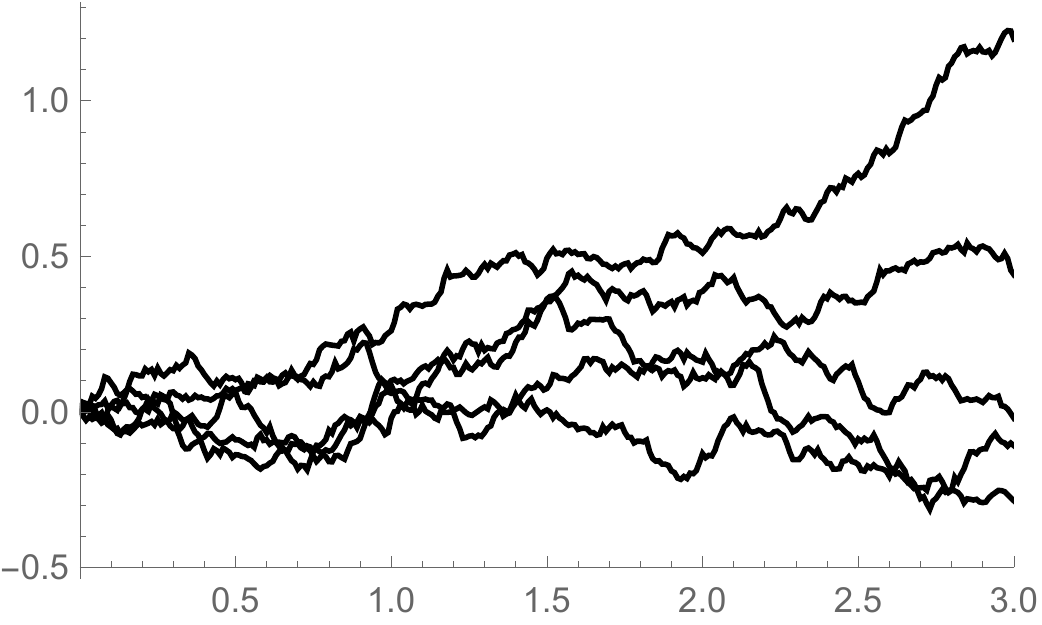}\tabularnewline
$Y_{t}^{\left(H\right)}$ & $Y_{t}^{\left(H\right)}$\tabularnewline
\hline 
 & \tabularnewline
\includegraphics[width=0.4\textwidth]{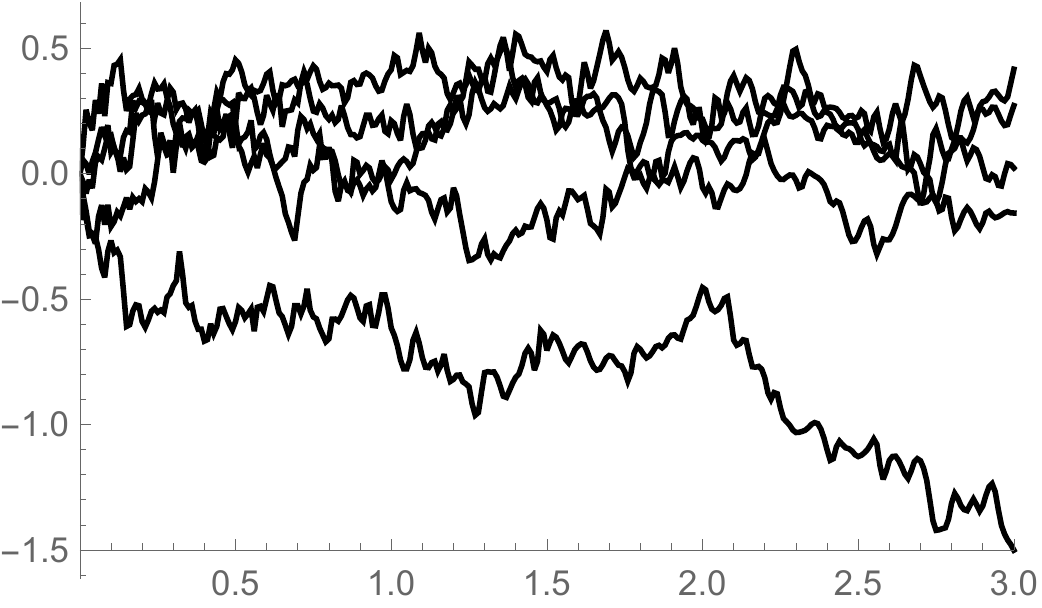} & \includegraphics[width=0.4\textwidth]{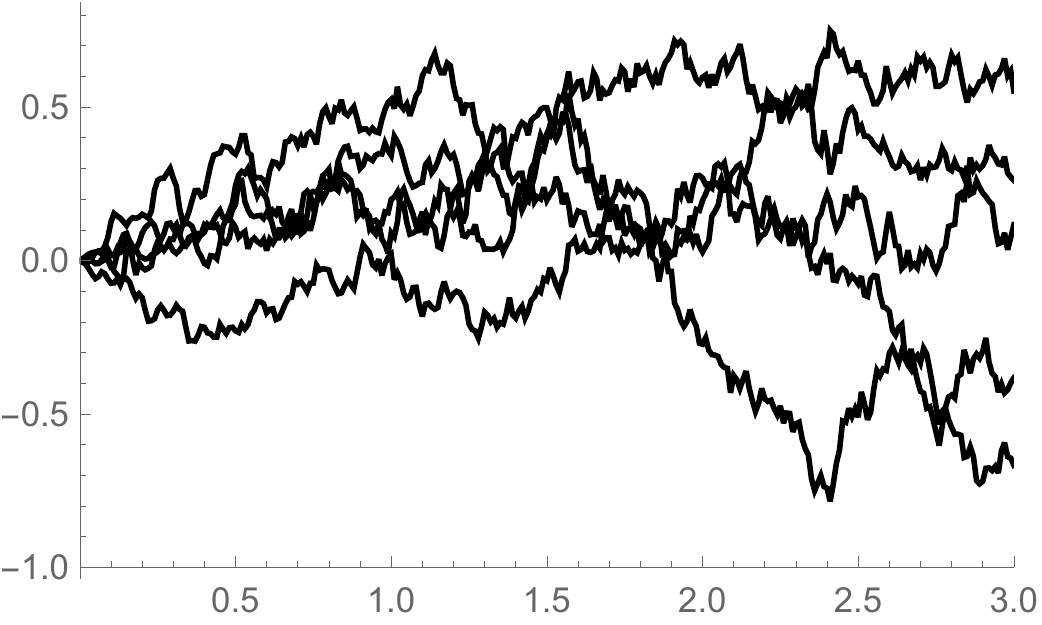}\tabularnewline
$X_{t}^{\left(H\right)}=B_{t^{2H}}$ & $X_{t}^{\left(H\right)}=B_{t^{2H}}$\tabularnewline
\hline 
\end{tabular}\vspace{5pt}

\caption{\label{fig:fBM}Fractional Brownian motion. The top two figures are
sample paths of fractional Brownian motion, while the bottom two are
the corresponding processes resulting from time change in the standard
Brownian motion.}
\end{figure}

\renewcommand{\arraystretch}{1}

The reason is that the two covariance kernels are difference. Indeed,
when $H\neq1/2$, 
\begin{equation}
\underset{{\scriptscriptstyle \mathbb{E}\left(X_{s}^{\left(H\right)}X_{t}^{\left(H\right)}\right)}}{\underbrace{\left(s\wedge t\right)^{2H}}}\neq\underset{{\scriptscriptstyle \mathbb{E}\left(Y_{s}^{\left(H\right)}Y_{t}^{\left(H\right)}\right)}}{\underbrace{\tfrac{1}{2}(s^{2H}+t^{2H}-\left|s-t\right|^{2H})}};\label{eq:M16}
\end{equation}
i.e., the two functions from (\ref{eq:M16}) are different on $J\times J$. 
\end{rem}

\begin{rem}
For general facts on fractional Brownian motion, and Hurst parameter,
see e.g., \cite{MR2793121} and \cite{MR1921068}.
\end{rem}

\section{Laplacians}

The purpose of the present section is to show that there is an important
class of Laplace operators, and associated energy Hilbert spaces $\mathscr{H}$,
which satisfies the conditions in our results from Sections \ref{sec:SigA}
and \ref{sec:SA} above. Starting with a fixed sigma-finite measure
$\mu$, the setting from sect \ref{sec:SA} entails pairs $(\beta,\mathscr{H})$,
subject to conditions (\ref{eq:S2}) and (\ref{eq:S4}), which admit
a certain spectral theory. With the condition in \thmref{T1}, we
showed that there are then induced sigma-finite measures $\mu_{f}$,
indexed by $f$ in a dense subspace in $\mathscr{H}$. The key consideration
implied by this is a closable, densely defined, operator $T$ from
$L^{2}(\mu)$ into $\mathscr{H}$. The induced measures $\mu_{f}$
are then indexed by $f$ in $dom(T^{*})$, the dense domain of the
adjoint operator $T^{*}$. If $\mathscr{H}$ is one of the energy
Hilbert spaces, then $T^{*}$ will be an associated Laplacian; see
details in \propref{L4}.

Now the Laplacians we introduce include variants from both discrete
network analysis, and more classical Laplacians from harmonic analysis.
As well as more abstract Laplacians arising in potential theory. There
is a third reason for the relevance of such new classes of Laplace-operators:
Each one of these Laplacians corresponds to a reversible Markov process
(and vice versa.) The latter interconnection will be addressed at
the end of section, but the more detailed implications, following
from it, will be postponed to future papers. As for the research literature,
it is fair to say that papers on reversible Markov processed far outnumber
those dealing with generalized Laplacians.

Let $\left(M,\mathscr{B},\mu\right)$ be a fixed $\sigma$-finite
positive measure, and let $\rho$ be a \emph{symmetric} positive measure
on the product space $\left(M\times M,\mathscr{B}_{2}\right)$ where
$\mathscr{B}_{2}$ denotes the product $\sigma$-algebra on $M\times M$,
i.e., the $\sigma$-algebra of subsets of $M\times M$ generated by
the cylinder sets 
\begin{equation}
\left\{ A\times B\mid A,B\in\mathscr{B}\right\} .
\end{equation}

We assume that $\rho$ admits a disintegration with $\mu$ as marginal
measure: 
\begin{equation}
d\rho\left(x,y\right)=\rho^{\left(x\right)}\left(dy\right)d\mu\left(x\right);\label{eq:L2}
\end{equation}
equivalently, 
\begin{equation}
\rho\left(A\times B\right)=\int_{A}\rho^{\left(x\right)}\left(B\right)d\mu\left(x\right),\label{eq:L3}
\end{equation}
$\forall A,B\in\mathscr{B}$. Note that since $\rho$ is symmetric,
we also have a field of measures $\rho^{\left(y\right)}\left(dx\right)$
such that 
\begin{equation}
\rho\left(A\times B\right)=\int_{B}\rho^{\left(y\right)}\left(A\right)d\mu\left(y\right).\label{eq:L4}
\end{equation}
For the theory of disintegration of measures, we refer to \cite{MR3441734,zbMATH06897817}
and the papers cited there.

Let $\pi_{i}$, $i=1,2$, denote the coordinate projections 
\[
\pi_{1}\left(x,y\right)=x,\quad\text{and}\quad\pi_{2}\left(x,y\right)=y,
\]
for $\left(x,y\right)\in M\times M$. Then from the assumptions above,
we get 
\[
\mu=\rho\circ\pi_{1}^{-1}=\rho\circ\pi_{2}^{-1}.
\]

We shall finally assume that
\begin{equation}
\rho\left(A\times M\right)<\infty,\quad\forall A\in\mathscr{B}_{fin}\label{eq:L5}
\end{equation}
where $\mathscr{B}_{fin}=\left\{ A\in\mathscr{B}\mid\mu\left(A\right)<\infty\right\} $;
and we set 
\begin{equation}
c\left(x\right)=\rho^{\left(x\right)}\left(M\right),\quad x\in M,\label{eq:L6}
\end{equation}
where the measures $\rho^{\left(x\right)}$ are the slice measures
from the disintegration formula (\ref{eq:L2}), or equivalently (\ref{eq:L3}).
We note that assumption (\ref{eq:L5}) may be relaxed. For the results
proved below, it will be enough to assume only that the function $c(x)$
defined by the RHS in (\ref{eq:L6}) be finite for almost all $x$,
so for a.a. $x$ with respect to the measure $\mu$. See (\ref{eq:L3}).

We shall need the measure $\nu$, given by 
\begin{equation}
d\nu\left(x\right)=c\left(x\right)d\mu\left(x\right).\label{eq:L7}
\end{equation}

Given a pair $\left(\mu,\rho\right)$, as above, set 
\begin{equation}
\left(Rf\right)\left(x\right)=\int_{M}f\left(y\right)\rho^{\left(x\right)}\left(dy\right),
\end{equation}
defined on all measurable functions $f$ on $\left(M,\mathscr{B}\right)$. 

The associated \emph{Laplacian} (\emph{Laplace operator}) is as follows:
\begin{align}
\left(\Delta f\right)\left(x\right) & =\int_{M}\left(f\left(x\right)-f\left(y\right)\right)\rho^{\left(x\right)}\left(dy\right)\nonumber \\
 & =c\left(x\right)f\left(x\right)-\left(Rf\right)\left(x\right).\label{eq:L9}
\end{align}

\begin{defn}
\label{def:T1}Let $\left(\mu,\rho\right)$ be as above, and let $\mathscr{E}$
be the associated energy Hilbert space consisting of measurable functions
$f$ on $\left(M,\mathscr{B}\right)$ such that 
\begin{equation}
\left\Vert f\right\Vert _{\mathscr{E}}^{2}=\frac{1}{2}\iint_{M\times M}\left|f\left(x\right)-f\left(y\right)\right|^{2}d\rho\left(x,y\right)<\infty;\label{eq:L10}
\end{equation}
modulo functions $f$ s.t. $\text{RHS}_{\left(\ref{eq:L10}\right)}=0$.
\end{defn}

\begin{lem}
\label{lem:L2}Let a fixed pair $\left(\mu,\rho\right)$ be as above;
and let $\nu$ be the induced measure on $\left(M,\mathscr{B}\right)$
given by (\ref{eq:L7}).
\begin{enumerate}
\item Then condition (\ref{eq:S2}) is satisfied for $\mathscr{H}=\mathscr{E}$
(the energy Hilbert space), and with 
\begin{equation}
\left\langle f,g\right\rangle _{\mathscr{E}}=\frac{1}{2}\iint_{M\times M}\left(f\left(x\right)-f\left(y\right)\right)\left(g\left(x\right)-g\left(y\right)\right)d\rho\left(x,y\right),\label{eq:L11}
\end{equation}
we have, for $A,B\in\mathscr{B}_{fin}$: 
\begin{equation}
\left\langle \chi_{A},\chi_{B}\right\rangle _{\mathscr{E}}=\nu\left(A\cap B\right)-\rho\left(A\times B\right),
\end{equation}
and for $A=B$, 
\begin{equation}
\left\Vert \chi_{A}\right\Vert _{\mathscr{E}}^{2}=\nu\left(A\right)-\rho\left(A\times A\right).
\end{equation}
\item If $\varphi\in\mathscr{D}_{fin}\left(\mu\right)$, and $f\in\mathscr{E}$,
then 
\begin{equation}
\left\langle \varphi,f\right\rangle _{\mathscr{E}}=\int_{M}\varphi\left(x\right)\left(\Delta f\right)\left(x\right)d\mu\left(x\right).\label{eq:T14}
\end{equation}
\end{enumerate}
\end{lem}

\begin{proof}[Proof sketch]
 Most of the assertions follow by direct computation, using the results
in Sections \ref{sec:SigA}--\ref{sec:SA} above; see also \cite{MR2159774,MR3441734,MR3701541,MR3630401,MR3796644},
and the papers cited there.
\end{proof}
\begin{cor}
Let the pair $\left(\mu,\rho\right)$ be as stated in \lemref{L2},
and let $\nu$ be the measure $d\nu\left(x\right)=c\left(x\right)d\mu\left(x\right)$
where $c\left(x\right)=\rho^{\left(x\right)}\left(M\right)$ as in
(\ref{eq:L7}). On $\mathscr{B}_{fin}\times\mathscr{B}_{fin}$, set
\begin{equation}
\beta\left(A,B\right)=\nu\left(A\cap B\right)-\rho\left(A\times B\right).
\end{equation}

Then $\beta$ is positive definite, and the corresponding RKHS $\mathscr{H}\left(\beta\right)$
naturally and isometrically, embeds as a closed subspace in the energy
Hilbert space $\mathscr{E}$ from (\ref{eq:L11}).
\end{cor}

\begin{prop}
\label{prop:L4}Let $\left(\mu,\rho\right)$ be as above, we denote
by $T$ the inclusion identification 
\begin{eqnarray}
\mathscr{D}_{fin}\left(\mu\right)\subset L^{2}\left(\mu\right) & \xrightarrow{\quad T\quad} & \mathscr{E}\\
\left(\varphi\in L^{2}\left(\mu\right)\right) & \xmapsto{\quad\phantom{T}\quad} & \left(\varphi\in\mathscr{E}\right).\nonumber 
\end{eqnarray}
\begin{enumerate}
\item Then $T$ is closable with respect to the respective inner products
in $L^{2}\left(\mu\right)$ and $\mathscr{E}$; see (\ref{eq:L11}). 

Moreover, for $f\in dom\left(T^{*}\right)\left(\subseteq_{dense}\mathscr{E}\right)$
we have 
\begin{equation}
T^{*}f=\Delta f
\end{equation}
where $\Delta$ is the Laplacian in (\ref{eq:L9}).
\item For $f\in dom\left(T^{*}\right)$, the induced measure $\mu_{f}$
from \secref{SA}, satisfies 
\begin{equation}
\mu_{f}\left(A\right)=\int_{A}\left(\Delta f\right)d\mu,\quad\forall A\in\mathscr{B}_{fin}.
\end{equation}
\end{enumerate}
\end{prop}

\begin{proof}
The details are essentially contained in the above.

It is convenient to derive the closability of $T$ as a consequence
of the following symmetry property: 

For operators $T$ and $T^{*}$, $L^{2}\left(\mu\right)\xrightarrow{\;T\;}\mathscr{E}$,
and $\mathscr{E}\xrightarrow{\;T^{*}\;}L^{2}\left(\mu\right)$, we
consider the following dense subspaces, respectively:
\begin{alignat}{1}
 & \mathscr{D}_{fin}\left(\mu\right)\subset L^{2}\left(\mu\right),\;\text{dense w.r.t. the \ensuremath{L^{2}\left(\mu\right)}-norm; and}\label{eq:T18}\\
 & \big\{ f\in\mathscr{E}\mid\Delta f\in L^{2}\left(\mu\right)\big\}\subset\mathscr{E},\;\text{dense w.r.t. the \ensuremath{\mathscr{E}}-norm \ensuremath{\left(\ref{eq:L10}\right)}.}\label{eq:T19}
\end{alignat}
Then a direct verification, using \lemref{L2}, (\ref{eq:L10})--(\ref{eq:T14}),
yields:
\begin{equation}
\langle\underset{{\scriptscriptstyle =\varphi}}{\underbrace{T\varphi}},f\rangle_{\mathscr{E}}=\left\langle \varphi,\Delta f\right\rangle _{L^{2}\left(\mu\right)}
\end{equation}
for all $\varphi\in dom\left(T\right)$ (see (\ref{eq:T18})), and
all $f\in dom\left(T^{*}\right)$ (see (\ref{eq:T19})). 

Equivalently, 
\begin{equation}
\left\langle \varphi,f\right\rangle _{\mathscr{E}}=\int_{M}\varphi\:\left(\Delta f\right)d\mu,\label{eq:T21}
\end{equation}
for functions $\varphi$ and $f$ in the respective domains. But we
already established (\ref{eq:T21}) in \lemref{L2} above; see (\ref{eq:T14}).
Now the conclusions in the Proposition follow.
\end{proof}
\begin{flushleft}
\textbf{Discrete time reversible Markov processes}
\par\end{flushleft}

Let $\left(M,\mathscr{B}\right)$ be a measure space. A Markov process
with state space $M$ is a stochastic process $\left\{ X_{n}\right\} _{n\in\mathbb{N}_{0}}$
having the property that, for all $n,k\in\mathbb{N}_{0}$, 
\begin{equation}
Prob\left(X_{n+k}\in A\mid X_{1},\cdots,X_{n}\right)=Prob\left(X_{n+k}\mid X_{n}\right)
\end{equation}
holds for all $A\in\mathscr{B}$. A Markov process is determined by
its transition probabilities 
\begin{equation}
P_{n}\left(x,A\right)=Prob\left(X_{n}\in A\mid X_{0}=x\right),
\end{equation}
indexed by $x\in M$, and $A\in\mathscr{B}$. 

It is known and easy to see that, if $\left\{ X_{n}\right\} _{n\in\mathbb{N}_{0}}$
is a Markov process, then 
\begin{equation}
P_{n+k}\left(x,A\right)=\int_{M}P_{n}\left(x,dy\right)P_{k}\left(y,A\right);
\end{equation}
and so, in particular, we have:
\begin{equation}
P_{n}\left(x,A\right)=\int_{y_{1}}\int_{y_{2}}\cdots\int_{y_{n-1}}P\left(x,dy_{1}\right)P\left(y_{1},dy_{2}\right)\cdots P\left(y_{n-1},A\right),
\end{equation}
for $x\in M$, $A\in\mathscr{B}$. 
\begin{defn}
\label{def:L5}Let $\mu$ be a $\sigma$-finite measure on $\left(M,\mathscr{B}\right)$.
We say that a Markov process is reversible iff there is a positive
measurable function $c$ on $M$ such that, for all $A,B\in\mathscr{B}$,
we have:
\begin{equation}
\int_{A}c\left(x\right)P\left(x,B\right)d\mu\left(x\right)=\int_{B}c\left(y\right)P\left(y,A\right)d\mu\left(y\right).
\end{equation}
\end{defn}

\begin{prop}
\label{prop:T6}Let $\left(M,\mathscr{B},\mu\right)$ be as usual,
and let $\left(P\left(x,\cdot\right)\right)$ be the generating transition
system for a Markov process. Then this Markov process $\left\{ X_{n}\right\} _{n\in\mathbb{N}_{0}}$
is reversible if and only if there is a positive measurable function
$c$ on $M$ such that the assignment $\rho$: 
\begin{equation}
\rho\left(A\times B\right)=\int_{A}c\left(x\right)P\left(x,B\right)d\mu\left(x\right),\quad A,B\in\mathscr{B},
\end{equation}
extends to a \uline{symmetric} sigma-additive positive measure
on the product $\sigma$-algebra $\mathscr{B}_{2}$, i.e., the $\sigma$-algebra
on $M\times M$ generated by product sets $\left\{ A\times B\mid A,B\in\mathscr{B}\right\} $. 
\end{prop}

\begin{proof}
The conclusion follows from the considerations above, and the remaining
details are left to the reader.
\end{proof}
\begin{cor}
\label{cor:T7}Let $\left(\mu,\rho\right)$ be a pair of measures,
$\mu$ on $\left(M,\mathscr{B}\right)$, $\rho$ on $\left(M\times M,\mathscr{B}_{2}\right)$
satisfying the conditions in (\ref{eq:L3})--(\ref{eq:L4}), and
let $c$ be the function from (\ref{eq:L6}), then 
\begin{equation}
P\left(x,A\right):=\frac{1}{c\left(x\right)}\rho^{\left(x\right)}\left(A\right)
\end{equation}
defines a reversible Markov process. 
\end{cor}

\begin{proof}
For measurable function $f$ on $\left(M,\mathscr{B}\right)$, i.e.,
$f:M\rightarrow\mathbb{R}$, set 
\[
\left(Pf\right)\left(x\right)=\int_{M}f\left(y\right)P\left(x,dy\right).
\]
Then the path space measure for the associated Markov-process $\left\{ X_{n}\right\} _{n\in\mathbb{N}_{0}}$
is determined by its conditional expectations evaluated on cylinder
functions: 
\begin{eqnarray*}
 &  & \mathbb{E}_{X_{0}=x}\left[f_{0}\left(X_{0}\right)f_{1}\left(X_{1}\right)f_{2}\left(X_{2}\right)\cdots f_{n}\left(X_{n}\right)\right]\\
 & = & f_{0}\left(x\right)P\left(f_{1}P\left(f_{2}\left(\cdots P\left(f_{n-1}P\left(f_{n}\right)\right)\right)\right)\cdots\right)\left(x\right).
\end{eqnarray*}
The result is now immediate from \defref{L5}.
\end{proof}
\begin{cor}
Let the pair $\left(\mu,\rho\right)$ be as above, and as in \lemref{L2}.
Let $\left\{ X_{n}\right\} _{n\in\mathbb{N}_{0}}$ be the corresponding
reversible Markov process; see \corref{T7}.
\begin{enumerate}
\item Then, for measurable functions $f$ on $\left(M,\mathscr{B}\right)$,
we have the following variance formula: 
\[
VAR_{X_{0}=x}\left(f\left(X_{1}\right)\right)=\int_{M}\left|f\left(y\right)-P\left(f\right)\left(x\right)\right|^{2}P\left(x,dy\right)
\]
\item Set $d\nu=c\left(x\right)d\mu\left(x\right)$, and let $\mathscr{E}$
denote the energy Hilbert space from \defref{T1}. Then a measurable
function $f$ on $\left(M,\mathscr{B}\right)$ is in $\mathscr{E}$
iff $f-P\left(f\right)\in L^{2}\left(\nu\right)$, and $VAR_{x}\left(f\left(X_{1}\right)\right)\in L^{1}\left(\nu\right)$.
In this case, 
\[
\left\Vert f\right\Vert _{\mathscr{E}}^{2}=\frac{1}{2}\left[\int_{M}\left|f-Pf\right|^{2}d\nu+\int_{M}VAR_{x}\left(f\left(X_{1}\right)\right)d\nu\left(x\right)\right].
\]
\end{enumerate}
\end{cor}

\begin{proof}
Immediate from the details in \propref{T6} and \corref{T7}.
\end{proof}
\begin{rem}
In the last section we pointed out the connection between reversible
Markov processes, and the Laplace operators, the energy Hilbert space,
and our results in Sections \ref{sec:SigA} and \ref{sec:SA}. However
we have postponed applications to reversible Markov processes to future
papers. For earlier papers regarding Laplace operators and associated
energy Hilbert space, see eg., \cite{MR3096586}. The literature on
reversible Markov processes is vast; see e.g., \cite{MR2584746,MR3638040,MR3441734,MR3530319}.
\end{rem}

\begin{acknowledgement*}
The present work was started during the NSF CBMS Conference, \textquotedblleft Harmonic
Analysis: Smooth and Non-Smooth\textquotedblright , held at the Iowa
State University, June 4--8, 2018, where the first named author gave
10 lectures. We thank the NSF for funding, the organizers, especially
Prof Eric Weber; as well as the CBMS participants, especially Profs
Daniel Alpay, and Sergii Bezuglyi, for many fruitful discussions on
the present topic, and for many suggestions. We are extremely grateful
to a referee who offered a number of excellent suggestions, helped
us broaden the list of pointers to additional applications of our
RKHS analysis; applications to yet more areas of probability theory,
and stochastic analysis. And finally, spotted places where corrections
were needed. We followed all suggestions. Indeed, his/her kind help
and suggestions much improved our paper.

\bibliographystyle{amsalpha}
\bibliography{ref}
\end{acknowledgement*}

\end{document}